\documentclass[12pt]{amsart}
\usepackage{latexsym, enumerate}
\usepackage{amssymb, tikz, mathrsfs, inputenc}

\usepackage{tikz}
\usetikzlibrary{shapes.geometric,arrows}

\tikzstyle{box} = [rectangle,text centered, draw=black]
\tikzstyle{arrow} = [thick, ->, >=stealth]

\usepackage{color}

\numberwithin{equation}{section}

\newtheorem{theorem}{Theorem}[section]

\newtheorem{definition}[theorem]{Definition}
\newtheorem{proposition}[theorem]{Proposition}
\newtheorem{corollary}[theorem]{Corollary}
\newtheorem{lemma}[theorem]{Lemma}

\newtheorem{remark}[theorem]{Remark}

\theoremstyle{definition}

\theoremstyle{remark}

\theoremstyle{assumption}

\title[]{Bergman spaces under maps of monomial type}
\authors
\author{Alexander Nagel}
\address{\hskip-0.15in Alexander Nagel \newline Department of Mathematics \newline  University of Wisconsin Madison \newline  Van Vleck Hall \newline  480 Lincoln Drive \newline Madison, WI 53706 USA}  
\email{ajnagel@wisc.edu}
\author{Malabika Pramanik}
\address{\hskip-0.15in Malabika Pramanik \newline Department of Mathematics \newline University of British Columbia \newline 1984 Mathematics Road \newline  Vancouver, BC V6T 1Z2 Canada}
\email{malabika@math.ubc.ca}

\date{\today}

\begin{document}
\def\a{\mathbf a}
\def\b{\mathbf b}
\def\c{\mathbf c}
\def\d{\mathbf d}
\def\e{\mathbf e}
\def\f{\mathbf f}
\def\g{\mathbf g}
\def\h{\mathbf h}
\def\i{\mathbf i}
\def\j{\mathbf j}
\def\k{\mathbf k}
\def\l{\mathbf l}
\def\m{\mathbf m}
\def\n{\mathbf n}
\def\o{\mathbf o}
\def\p{\mathbf p}
\def\q{\mathbf q}
\def\r{\mathbf r}
\def\s{\mathbf s}
\def\t{\mathbf t}
\def\u{\mathbf u}
\def\v{\mathbf v}
\def\w{\mathbf w}
\def\x{\mathbf x}
\def\y{\mathbf y}
\def\z{\mathbf z}

\def\0{{\mathbf 0}}
\def\1{{\mathbf 1}}
\def\2{{\mathbf 2}}
\def\3{{\mathbf 3}}
\def\4{{\mathbf 4}}
\def\5{{\mathbf 5}}
\def\6{{\mathbf 6}}
\def\7{{\mathbf 7}}
\def\8{{\mathbf 8}}
\def\9{{\mathbf 9}}

\def\A{\mathbb A}       
\def\B{\mathbb B}
\def\C{\mathbb C}
\def\D{\mathbb D}
\def\E{\mathbb E}
\def\F{\mathbb F}
\def\G{\mathbb G}
\def\H{\mathbb H}
\def\I{\mathbb I}
\def\J{\mathbb J}
\def\K{\mathbb K}
\def\L{\mathbb L}
\def\M{\mathbb M}
\def\N{\mathbb N}
\def\O{\mathbb O}
\def\P{\mathbb P}
\def\Q{\mathbb Q}
\def\R{\mathbb R}
\def\S{\mathbb S}
\def\T{\mathbb T}
\def\U{\mathbb U}
\def\V{\mathbb V}
\def\W{\mathbb W}
\def\X{\mathbb X}
\def\Y{\mathbb Y}
\def\Z{\mathbb Z}

\def\AA{\mathcal A}
\def\BB{\mathcal B}
\def\CC{\mathcal C}
\def\DD{\mathcal D}
\def\EE{\mathcal E}
\def\FF{\mathcal F}
\def\GG{\mathcal G}
\def\HH{\mathcal H}
\def\II{\mathcal I}
\def\JJ{\mathcal J}
\def\KK{\mathcal K}
\def\LL{\mathcal L}
\def\MM{\mathcal M}
\def\NN{\mathcal N}
\def\OO{\mathcal O}
\def\PP{\mathcal P}
\def\QQ{\mathcal Q}
\def\RR{\mathcal R}
\def\SS{\mathcal S}
\def\TT{\mathcal t}
\def\UU{\mathcal U}
\def\VV{\mathcal V}
\def\WW{\mathcal W}
\def\XX{\mathcal X}
\def\YY{\mathcal Y}
\def\ZZ{\mathcal Z}

\def \bfA {\mathbf A}
\def \bfB {\mathbf B}
\def \bfC {\mathbf C}
\def \bfD {\mathbf D}
\def \bfE {\mathbf E}
\def \bfF {\mathbf F}
\def \bfG {\mathbf G}
\def \bfH {\mathbf H}
\def \bfI {\mathbf I}
\def \bfJ {\mathbf J}
\def \bfK {\mathbf K}
\def \bfL {\mathbf L}
\def \bfM {\mathbf M}
\def \bfN {\mathbf N}
\def \bfO {\mathbf O}
\def \bfP {\mathbf P}
\def \bfQ {\mathbf Q}
\def \bfR {\mathbf R}
\def \bfS {\mathbf S}
\def \bfR {\mathbf R}
\def \bfR {\mathbf R}
\def \bfT {\mathbf T}
\def \bfU {\mathbf U}
\def \bfV {\mathbf V}
\def \bfW {\mathbf W}
\def \bfX {\mathbf X}
\def \bfY {\mathbf Y}
\def \bfZ {\mathbf Z}

\def\CZ{Calder\'on-Zygmund }

\def\alphab{\boldsymbol{\alpha}}
\def\betab{\boldsymbol{\beta}}
\def\gammab{\boldsymbol{\gamma}}
\def\deltab{\boldsymbol{\delta}}
\def\epsilonb{\boldsymbol{\epsilon}}
\def\zetab{\boldsymbol{\zeta}}
\def\etab{\boldsymbol{\eta}}
\def\thetab{\boldsymbol{\theta}}
\def\iotab{\boldsymbol{\iota}}
\def\kappab{\boldsymbol{\kappa}}
\def\lambdab{\boldsymbol{\lambda}}
\def\mub{\boldsymbol{\mu}}
\def\nub{\boldsymbol{\nu}}
\def\xib{\boldsymbol{\xi}}
\def\pib{\boldsymbol{\pi}}
\def\rhob{\boldsymbol{\rho}}
\def\sigmab{\boldsymbol{\sigma}}
\def\taub{\boldsymbol{\tau}}
\def\upsilonb{\boldsymbol{\upsilon}}
\def\phib{\boldsymbol{\phi}}
\def\varphib{\boldsymbol{\varphi}}
\def\chib{\boldsymbol{\chi}}
\def\psib{\boldsymbol{\psi}}
\def\omegab{\boldsymbol{\omega}}
\def\varthetab{\boldsymbol{\vartheta}}

\def\bigkappa{\text{\LARGE $\kappa$}}

\def\dbar{\bar\partial}
\def\bx{\square}
\def \RE {\Re\text{\rm e}}
\def \IM {\Im\text{\rm m}}
\def\bea[{[\![}
\def\eea]{]\!]}

\def\interior{\text{int\,}}
\newcommand\bkt[2]{\left\langle #1,#2\right\rangle}

\def\be{\begin{equation}}
\def\ee{\end{equation}}
\def\bes{\begin{equation*}}
\def\ees{\end{equation*}}
\def\bea{\begin{equation}\begin{aligned}}
\def\eea{\end{aligned}\end{equation}}
\def\beas{\begin{equation*}\begin{aligned}}
\def\eeas{\end{aligned}\end{equation*}}

\def\AAA{\mathbf A}
\def\BBB{\mathbf B}
\def\CCC{\mathbf C}
\def\DDD{\mathbf D}
\def\EEE{\mathbf E}
\def\EEEt{\tidetilde{\EEE}}
\def\FFF{\mathbf F}
\def\MMM{\mathbf M}
\def\SSS{\mathbf S}
\def\TTT{\mathbf T}
\def\tI{\tilde I}
\def\xt{\tilde{\x}}
\def\cdott{\,\tilde\cdot\,}
\def\INT{\text{int}}
\def\VOL{\text{Vol}}

{\allowdisplaybreaks

\begin{abstract}
For appropriate domains $\Omega_{1}, \Omega_{2}$ we consider mappings $\Phi_{\mathbf A}:\Omega_{1}\to\Omega_{2}$ of monomial type. We obtain an orthogonal decomposition of the Bergman space $\mathcal A^{2}(\Omega_{1})$ into finitely many closed subspaces indexed by characters of a finite Abelian group associated to the mapping $\Phi_{\mathbf A}$. We then show that each subspace is isomorphic to a weighted Bergman space on $\Omega_{2}$. This leads to a formula for the Bergman kernel on $\Omega_{1}$ as a sum of weighted Bergman kernels on $\Omega_{2}$. 
\end{abstract}
\thanks{Part of this work was finalized in April 2019, at the Banff International Research Station (BIRS) in Banff, Alberta during a ``Research in Teams" residency program. The authors are grateful to BIRS for their hospitality and support during this stay. MP would like to thank Prof.\ Jonathan Pakianathan for a helpful discussion at an initial stage of the project, regarding the material of Section \ref{Sec3} and in particular for indicating the reference \cite{Munkres}.  AN was supported in part by funds from a Steenbock Professorship at the University of Wisconsin-Madison. MP was partially supported through NSERC Discovery grants and a Wall Scholarship from the Peter Wall Institute for Advanced Study. } 

\maketitle

\thispagestyle{empty}
 
\section{Introduction} 

\noindent Let $\Omega_{1}, \,\Omega_{2}\subseteq\mathbb C^{n}$ be open sets and let $\Phi_{\mathbf A}:\Omega_{1}\to\Omega_{2}$ be a surjective holomorphic mapping of monomial type. In this paper we obtain a decomposition of weighted Bergman spaces on $\Omega_{1}$ associated to the mapping $\Phi_{\mathbf A}$ as well as relationships between the weighted Bergman kernels of $\Omega_{1}$ and $\Omega_{2}$. In this Introduction we begin by recalling the definitions of these concepts, and then state our main results.

\subsection{Bergman projections and kernels}

\noindent Let $\Omega$ be an open set in $\mathbb C^n$, $n \geq 1$ with Lebesgue measure $dV$. Given a continuous \emph{weight function} $\omega: \Omega \rightarrow (0, \infty)$, denote by $\mathcal L^2(\Omega, \omega)$ the Hilbert space of (equivalence classes of) Lebesgue-measurable functions on $\Omega$ that are square-integrable with respect to the measure $\omega(\mathbf z) dV(\mathbf z)$. The closed subspace $\mathcal A^2(\Omega, \omega)\subseteq\mathcal L^2(\Omega, \omega)$ consisting of functions that are holomorphic on $\Omega$ is the corresponding {\em{weighted Bergman space}}. The orthogonal projection $\mathscr P_{\Omega}^{\omega}:\mathcal L^2(\Omega, \omega)\to \mathcal A^2(\Omega, \omega)$ is the {\em{weighted Bergman projection}}. For $f\in \mathcal L^{2}(\Omega, \omega)$ and $\z\in \Omega$ the projection $\mathscr P^{\omega}_{\Omega}[f]$ is given by
\begin{equation} \label{Eqn1.1}
\mathscr P^{\omega}_{\Omega}[f](\mathbf z) = \int_{\Omega} B_{\Omega}(\mathbf z, \mathbf w; \omega) f(\mathbf w) \omega(\mathbf w) dV(\mathbf w).  
\end{equation} 
The integration kernel $B_{\Omega}(\cdot, \cdot; \omega): \Omega \times \Omega \rightarrow \mathbb C$ is the {\em{weighted Bergman kernel}}. If $\{\psi_{j} : j \geq 1\}$ is any complete orthonormal basis for $\mathcal A^{2}(\Omega, \omega)$ then 
\begin{equation} \label{Eqn1.2}
B_{\Omega}(\mathbf z,\mathbf w; \omega)=\sum_{j=1}^{\infty}\psi_{j}(\mathbf z)\overline{\psi_{j}(\mathbf w)},
\end{equation} 
where the series converges absolutely and uniformly on compact subsets of $\Omega\times\Omega$. The value of the Bergman kernel when $\mathbf z=\mathbf w$ is the solution of an extremal problem:
\begin{equation} \label{Eqn1.3a}
B_{\Omega}(\mathbf z,\mathbf z; \omega)=\sup\Big\{|h(\mathbf z)|^{2}:\text{$h\in \mathcal A^{2}(\Omega,\omega)$ and $||h||_{2, \omega}\leq 1$}\Big\}, 
\end{equation} 
where $|| \cdot ||_{2, \omega}$ denotes the norm in $\mathcal L^2(\Omega, \omega)$. It follows that 
\begin{equation} \label{Eqn1.4a}
\Omega_1 \subseteq \Omega_2\quad\Longrightarrow\quad B_{\Omega_{2}}(\mathbf z,\mathbf z; \omega)\leq B_{\Omega_{1}}(\mathbf z,\mathbf z; \omega)
\end{equation} 
for all $\z\in\Omega_{1}$. See \cite{Krantz} for the basic facts about the Bergman kernel and projection. We often omit $\omega$ when $\omega \equiv 1$, in which case $\mathcal A^2(\Omega)$ and $B_{\Omega}$ are referred to respectively as the {\em{standard Bergman space and standard Bergman kernel}} of $\Omega$. 

\vskip0.1in
 \noindent In this paper we are concerned with one aspect of the following general question:  
  
\vskip0.1in
\begin{quote} 
{\em{If $\Omega_1$, $\Omega_2\subseteq\mathbb C^n$ are open and  $\Phi: \Omega_1 \rightarrow \Omega_2$ is a surjective holomorphic mapping, how are weighted Bergman spaces on $\Omega_1$ related to those on $\Omega_2$?}}
\end{quote}
\vskip0.1in 
When $\omega \equiv 1$ and $\Phi:\Omega_{1}\to\Omega_{2}$ is biholomorphic, the answer to the above question is well-known. Specifically, we have that $\int_{\Omega_{2}}f(\mathbf w)\,dV(\mathbf w)=\int_{\Omega_{1}}f\big(\Phi(\mathbf z)\big)|\det J\Phi(\mathbf z)|^{2}\,dV(\mathbf z)$  for every $f\in L^{1}(\Omega_{2})$ where $J\Phi$ is the complex Jacobian matrix of $\Phi$. Since $\det J\Phi(\mathbf z)$ is nonvanishing and holomorphic, it follows  that 
 \begin{align}
 \mathscr{P}_{\Omega_1} \bigl( \bigl[ \det J \Phi \bigr] \cdot \bigl[f \circ \Phi \bigr] \bigr) &= \bigl[ \det J \Phi \bigr] \cdot \bigl[ \mathscr{P}_{\Omega_2} f \circ \Phi \bigr],  \label{transformation-law}  \\ 
 \int_{\Omega_{1}}B_{\Omega_{1}}(\mathbf z, \mathbf u)f\big(\Phi(\mathbf u)\big) \det J\Phi(\mathbf u)\,d\mathbf u
&=
\det J\Phi(\mathbf z)\int_{\Omega_{2}}B_{\Omega_{2}}(\Phi(\mathbf z),\mathbf v)f(\mathbf v)\,d\mathbf v.\label{Eqn1.5} \end{align}
Since $\mathbf v = \Phi(\mathbf w)$ for a unique $\mathbf w \in \Omega_1$, it follows from \eqref{Eqn1.5}  that  
\begin{equation} 
  B_{\Omega_1}(\mathbf z, \mathbf w) = \bigl[ \det J \Phi(\mathbf z) \bigr] \bigl[ B_{\Omega_2}(\Phi(\mathbf z), \Phi(\mathbf w)) \bigr]  \bigl[\overline{\det J \Phi(\mathbf w)} \bigr]. \label{transformation-law2}
  \end{equation}
For details see for example \cite{Krantz}, Proposition 1.4.12, page 52.  
    
\vskip0.1in

\noindent There has been considerable previous work which has resulted in generalizations of the formulas in \eqref{transformation-law}, \eqref{Eqn1.5},  and \eqref{transformation-law2}. In \cite{Bell} and \cite{Bell2}  Steven Bell generalized these results by showing that equations (\ref{transformation-law}) and (\ref{Eqn1.5}) continue to hold if each $\Omega_{i}$ is a bounded domain and $\Phi: \Omega_1 \rightarrow \Omega_2$ is a \emph{proper} holomorphic mapping (\emph{i.e.} $\Phi^{-1}(K)\subset\Omega_{1}$ is compact for each compact subset $K\subset\Omega_{2})$. Proper mappings are finite branched coverings. If $\Phi$ is an $m$-fold branched covering let $\Psi_{1}, \ldots, \Psi_{m}$ denote the $m$ local inverses of $\Phi$. In this case the identity (\ref{transformation-law2}) is replaced by the formula
\begin{equation}\label{Eqn1.7}
\sum_{i=1}^{m}B_{\Omega_{1}}\big(\mathbf z, \Psi_{i}(\mathbf v)\big)\overline{\det J\Psi_{i}(\mathbf v)}= \det J\Phi(\mathbf z)\,B_{\Omega_{2}}(\Phi(\mathbf z),\mathbf v), \quad \mathbf z \in \Omega_1, \; \mathbf v \in \Omega_2. 
\end{equation}

\vskip0.1in

\noindent In another direction, Siqi Fu \cite{Fu}, using the Poisson summation formula,  established similar formulas in certain cases of infinite covering maps from tube domains to Reinhardt domains.

\vskip0.1in

\noindent In Section \ref{Sec7.2} below we illustrate our results for the domain \[ \Omega=\big\{(z_{1}, z_{2})\in \C^{2}:0<|z_{1}|^{p}<|z_{2}|^{q}<1\big\},  \] with $p$ and $q$ positive integers, which is a generalization of the Hartogs triangle $\Omega = \big\{(z_{1}, z_{2})\in \C^{2}:0<|z_{1}|<|z_{2}|<1\big\}$. We are grateful to Jeff McNeal and Debraj Chakrabarti for bringing our attention to recent work \cite{{HM2015}, {HM2017}, {CEM2019},  {CKMM2020}, {CH2020}} related to this example.  For example, Edholm and McNeal have studied the Bergman kernel for domains related to the ``fat Hartogs triangle'' \[ \Omega_{\gamma}=\big\{(z_{1},z_{2})\in \C^{2}:|z_{1}|^{\gamma}<|z_{2}|<1\big\} \quad  \text{ where } \gamma>0. \] In \cite{HM2015} they show that if $\gamma$ is a positive integer then the Bergman kernel is bounded on $L^{r}(\Omega_{\gamma})$ if and only if $r \in \left((2\gamma+2)/{(\gamma+2)}, (2\gamma+2)/\gamma \right)$. In \cite{HM2017} they show that if $\gamma$ is irrational then the Bergman kernel is bounded on $L^{r}(\Omega_{\gamma})$ only when $r=2$. Chakrabarti, Edholm, and McNeal \cite{CEM2019} study duality and approximation issues on these and more general bounded Reinhardt domains. Chakrabarti, Konkel, Mainkar, and Miller \cite{CKMM2020} calculate the Bergman kernel for more general domains $\Omega_{\mathbf k}=\big\{(z_{1}, \ldots, z_{n})\in \C^{n}:|z_{1}|^{k_{1}}\cdots|z_{n}|^{k_{n}}<1\big\}$ where the exponents $k_{j}$ are (possibly negative) integers. In these papers the authors decompose the Bergman kernel for a domain into ``sub-Bergman kernels'' and these are related to the decompositions we make. Chakrabarti and Edholm \cite{CH2020} study the relationship between the $L^{r}$-mapping properties of the Bergman kernels on two domains one of which is the quotient of the other, under the action of a finite group of biholomorphic automorphisms.

\medskip

\noindent There is an extensive literature addressing the fundamental role of the Bergman projection and its kernel in complex function theory. This paper is one in a series of work by the authors \cite{{NP-2009}, {NP-Stein-conference}, {NP-preprint}, {NP-preprint-Reinhardt-general}, {NP-preprint-Reinhardt-monomial}} dealing with estimates for the Bergman kernel in various domains. Our results in this paper are motivated by our interest in estimates for \emph{complex monomial balls}, discussed in Section  \ref{appendix-section} below. Our results and objectives are of a different nature than in the earlier work of Bell \cite{Bell}, \cite{Bell2},  and are based on the algebraic structure of the mapping $\Phi_{\mathbf A}$.  

\subsection{Monomial mappings} \smallskip In this paper, we consider mappings $\Phi_{\mathbf A}$ and functions $F_{\mathbf b}$ of {\em{monomial type}}. If $\mathbf b = (b_1, \cdots, b_n) \in \mathbb Z^n$, if $\mathbf A=\{a_{j,k}\}$ is a non-singular $n \times n$ matrix with integer entries, and if $\mathbf z = (z_1, \cdots, z_n) \in \mathbb C^n$, we set  
\begin{equation} 
F_{\mathbf b} (\mathbf z) := z_1^{b_{1}} z_2^{b_{2}} \cdots z_n^{b_{n}}, \qquad   
\label{def-Phi}
\Phi_{\mathbf A}(\mathbf z) := (F_{\mathbf a_1}(\mathbf z), \cdots, F_{\mathbf a_n}(\mathbf z)),
\end{equation}
where $\mathbf a_j = (a_{j,1}, \cdots, a_{j,n})$ denotes the $j^{\text{th}}$ row vector of $\mathbf A$. If all the entries of the matrix $\mathbf A$ are non-negative integers, then $\Phi_{\mathbf A}$ is holomorphic on all of $\mathbb C^n$. If $\mathbf A$ has at least one negative entry, then $\Phi_{\mathbf A}$ is holomorphic at $\mathbf z$ if and only if  
$\z \in \mathbb C^n \setminus \mathbb H_{\mathbf A}$ where  \[ \mathbb H_{\mathbf A} =  \bigcup_{k=1}^n \Bigl\{\mathbf z \in \mathbb C^n : z_k = 0, \; \text{and there exists }  1 \leq j \leq n \text{ such that } a_{j,k} < 0 \Bigr\}.  \] 
In particular, for any non-singular $n \times n$ matrix with arbitrary integer entries, the mapping $\Phi_{\mathbf A}$ is always holomorphic on $\mathbb C_{\ast}^n := \mathbb C^n \setminus \mathbb H$, where $\mathbb H$ is the union of coordinate hyperplanes: 
\begin{equation} \label{coord-hyperplanes} \mathbb H:= \Bigl\{\mathbf z = (z_1, \cdots, z_n) \in \mathbb C^n : z_{1}z_{2}\cdots z_{n} = 0 \Bigr\}. \end{equation}
For any integer-valued matrix $\mathbf A$, the Jacobian of $\Phi_{\mathbf A}$ can be singular only at points in $\mathbb H$. Basic properties of monomial type functions and mappings are presented in Section \ref{Sec-Monomial-Map}.

\smallskip

\subsection{The groups $\mathbb G_{\mathbf A}$ and $\widehat{\mathbb G}_{\mathbf A}$}  
\noindent We now introduce algebraic objects associated with monomial mappings. In this paper, all vectors in $\mathbb R^n$ are considered row vectors, i.e., $1 \times n$ matrices. Matrix multiplication is denoted by ``$\cdot$".  $\mathbb M_n(\mathbb Z)$ and $\mathbb M_n(\mathbb R)$  denote the spaces of $n \times n$ matrices with integer and real entries respectively. The transpose and inverse of a matrix $\mathbf M$ are denoted by $\mathbf M^t$ and $\mathbf M^{-1}$. The notation $\langle \cdot, \cdot \rangle$ stands for the real inner product, \emph{i.e.} if $\mathbf z = (z_1, \cdots, z_n), \; \mathbf w = (w_1, \cdots, w_n)\in\mathbb C^{n}$ then $\langle \mathbf z, \mathbf w \rangle := \sum_{j=1}^{n} z_j w_j$. Let $\e_{1}, \ldots, \e_{n}$ denote the standard basis elements of $\R^{n}$. 

\begin{definition}\label{Def1.1}\quad

\begin{enumerate}[(a)]
\item If $\mathbf A\in \mathbb M_{n}(\mathbb Z)$ then $\mathfrak C(\mathbf A) :=\big\{\mathbf m\cdot\mathbf A^{t}:\mathbf m\in \mathbb Z^{n}\big\}$ denotes the $\mathbb Z$-submodule of $\mathbb Z^{n}$ generated by the columns of $\mathbf A$, and $\mathfrak C(\mathbf A^{t}) :=\big\{\m\cdot\mathbf A:\m\in \Z^{n}\big\}$ denotes the $\mathbb Z$-submodule of $\mathbb Z^{n}$ generated by the rows of $\mathbf A$.

\smallskip

\item $\mathbb G_{\mathbf A} := \mathbb Z^{n}/\mathfrak C(\mathbf A)$ and $\mathbb G_{\mathbf A^t} := \mathbb Z^{n}/\mathfrak C(\mathbf A^t)$ denote the quotient  groups; if $\mathbf m\in \mathbb Z^{n}$ then $[\mathbf m]$ denotes its equivalence class in  $\mathbb G_{\mathbf A}$ and $[\![\m]\!]$ denotes its equivalence class in $\mathbb G_{\mathbf A^t}$.

\smallskip

\item  \label{group-action}  If $[\mathbf m] \in \mathbb G_{\mathbf A}$  set 
\begin{equation} \label{Eqn1.10}
\xi_j ([\mathbf m]) := \exp \bigl[ 2\pi i \langle\m,\e_{j}\cdot \mathbf A^{-1}\rangle\bigr] \quad \text{  and }  \quad \xib([\m]) :=\bigl(\xi_1([\mathbf m]), \cdots, \xi_n([\mathbf m]) \bigr).
\end{equation} 
\smallskip

\item If $\v=(v_{1}, \ldots, v_{n}),\,\w=(w_{1}, \ldots, w_{n})\in \mathbb C^{n}$ then $\v\otimes\w=(v_{1}w_{1}, \ldots, v_{n}w_{n})$ denotes the Hadamard vector product.

\smallskip

\item  $\widehat{\G}_{\mathbf A}$ denote the group of characters of $\G_{\mathbf A}$, \emph{i.e.} the set of group homomorphisms from $\mathbb G_{\mathbf A}$ to the unit circle $\T=\big\{z\in \C:|z|=1\big\}$, equipped with point-wise multiplication. An element of $\widehat{\mathbb G}_{\mathbf A}$ is thus a map $\chi:\G_{\mathbf A}\to \T$ such that $\chi([\mathbf m]+[\mathbf n]) = \chi([\mathbf m]) \chi([\mathbf n])$.

\smallskip

\item \label{chi-b} If $\b\in \Z^{n}$ the function $\chi_{\b} : \mathbb G_{\mathbf A} \rightarrow \mathbb T$ given by $\chi_{\b}([\m]) :=\exp\big[2\pi i \langle\m,\b\cdot\mathbf A^{-1}\rangle\big]$ is a character of $\G_{\mathbf A}$.  

\end{enumerate}
\end{definition}
\noindent In Section \ref{Sec-Monomial-Map} we study  the algebraic structure of $\mathfrak C(\mathbf A)$, $\mathbb G_{\mathbf A}$, and $\widehat{\G}_{\mathbf A}$. We see that  $\mathbb G_{\mathbf A}$ and $\widehat{\mathbb G}_{\mathbf A}$ are finite abelian groups of order $\det(\mathbf A)$.  We also  show that the binary operation 
\begin{equation} \label{Eqn1.11}
([\m],\z)\to \xib([\m])\otimes \mathbf z = \Bigl(e^{2\pi i \langle\m,\e_{1}\cdot \mathbf A^{-1}\rangle} z_1, \ldots, e^{2\pi i \langle\m,\e_{n}\cdot \mathbf A^{-1}\rangle}z_n \Bigr)
\end{equation} 
is a a \emph{faithful} action of $\G_{\mathbf A}$  on $\C^{n}_{*}$.

\smallskip

\noindent Note that if $\b_{1}, \b_{2}\in \mathbb Z^{n}$ then the charcters $\chi_{\b_{1}}, \chi_{\b_{2}}$ defined in part (\ref{chi-b}) of Definition \ref{Def1.1} are equal if and only if $\b_{1}-\b_{2}=\n\cdot\mathbf A$ for some $\n\in \Z^{n}$; i.e. if and only if $\b_{1}-\b_{2}\in \mathfrak C(\mathbf A^{t})$. The correspondence $[\![\mathbf b]\!] \mapsto \chi_{\mathbf b}$ therefore generates a mapping $\varphi:\G_{\mathbf A^{t}}\to\widehat{\G}_{\mathbf A}$
\bea\label{Eqn1.17}
\varphi([\![\b]\!])([\m]) :=\chi_{\b}([\m])=\exp\big[2\pi i \langle\m,\b\cdot\mathbf A^{-1}\rangle\big].
\eea
Lemma \ref{Lem3.5} below shows that $\varphi$ defined in \eqref{Eqn1.17} is a group isomorphism. Thus the characters of $\G_{\mathbf A}$ are parameterized by elements of the group $\G_{\mathbf A^{t}}$.

\subsection{Invariant domains and orthogonal decompositions}
An open set $\Omega_{1}\subseteq \C^{n}$ is said to be {\em{invariant}} \label{domain-invariance} under the action of $\mathbb G_{\mathbf A}$ defined in equation \eqref{Eqn1.11} if for every $\mathbf z \in \Omega_1^{\ast} = \Omega_1 \cap \mathbb C_{\ast}^n$ and every $\mathbf m \in \mathbb Z^n$, the point $\pmb{\xi}([\mathbf m]) \otimes \mathbf z$ is also in $\Omega_1^{\ast}$. A function $f: \Omega_1 \rightarrow \mathbb C$ is said to be {\em{invariant}} under this group action  if 
\begin{equation}
f \bigl( \xib ([\mathbf m]) \otimes \mathbf z \bigr) = f(\mathbf z) \text{ for all } \mathbf z \in \Omega_1^{\ast} \text{ and for all }[\mathbf m]\in \mathbb G_{\mathbf A}.  \label{function-invariance}
\end{equation} 
\noindent Suppose that $\Omega_{1}\subseteq\mathbb C^{n}$ is invariant under the action of $\mathbb G_{\mathbf A}$. For each $\chi\in \widehat{\G}_{\AAA}$ and any function $f:\Omega_{1}\to \C$,  define
\bea\label{Eqn1.12}
\Pi_{\chi}[f](\z) := \frac{1}{\#(\mathbb G_{\mathbf A})}\sum_{[\m]\in \G_{\mathbf A}}\chi([\m])f\big(\xib([\m]\otimes \z)\big)
\eea
for $\mathbf z\in\Omega_{1}$. In Section \ref{Sec4} we will show the following.
\begin{enumerate}[$\bullet$]
\smallskip
\item 
Each $\Pi_{\chi}$ is a projection: $\Pi_{\chi}^{2}=\Pi_{\chi}$ and  $f(\z)=\sum_{\chi\in \widehat{\G}_{\AAA}}\Pi_{\chi}[f](\z)$ for $z\in \Omega_{1}$.

\smallskip

\item If $\b\in \Z^{n}$ and $\chi_{\b}$ is the character given in part (\ref{chi-b}) of Definition \ref{Def1.1} then for all $\z\in \Omega_{1}$
\beas
\Pi_{\chi_{\b}}[f]\big(\xib([\m])\otimes \z\big)&=\chi_{\b}([\m])^{-1}\,\Pi_{\chi_{\b}}[f](\z), \qquad F_{\b}(\xib([\m])\otimes\z)&=\chi_{\b}([\m])\, F_{\b}(\z).
\eeas

\smallskip

\item The function $\Pi_{\chi_{\b}}[f](\cdot)\,F_{\b}(\cdot)$ is invariant under the action of $\G_{\AAA}$:
\begin{equation} \label{invariance-Pi-chi}  \Pi_{\chi_{\b}}[f]\big(\xib([\m])\otimes \z\big)\,F_{\b}(\xib([\m])\otimes\z)=\Pi_{\chi_{\b}}[f](\z)F_{\b}(\z)
\end{equation} 
\end{enumerate}

\noindent These observations lead to the following orthogonal decompositions of the spaces $\mathcal L^{2}(\Omega_{1};\omega_{1})$ and $\mathcal A^{2}(\Omega_{1};\omega_{1})$, parameterized by the characters of $\G_{\AAA}$.

\begin{theorem} \label{Thm1.1}
Let $\AAA\in\M_{n}(\Z)$ be non-singular, possibly with negative entries. Let $\Omega_{1}\subset\C^{n}$ be an open set and  let $\omega_{1}:\Omega_{1}\to (0,\infty)$ be continuous, both invariant under the action of $\G_{\AAA}$.  For each character $\chi\in \widehat{\G}_{\AAA}$, let $\Pi_{\chi}$ be the projection operator defined in equation \eqref{Eqn1.12}. Then the following conclusions hold. 

\begin{enumerate}[(a)]

\smallskip

\item \label{Thm2.1d}
The mapping $\Pi_{\chi}$ acting  on $\mathcal L^{2}(\Omega_{1};\omega_{1})$ or  $\mathcal A^{2}(\Omega_{1};\omega_{1})$ is an orthogonal projection.

\medskip

\item \label{Thm2.1e} Denote by $\mathcal L^{2}_{\chi}(\Omega_{1};\omega_{1}) :=\Pi_{\chi}\big[\mathcal L^{2}(\Omega_{1};\omega_{1})\big]$ and $\mathcal A^{2}_{\chi}(\Omega_{1};\omega_{1}) :=\Pi_{\chi}\big[\mathcal A^{2}(\Omega_{1};\omega_{1})\big]$ the ranges of the projection $\Pi_{\chi}$. Then the following are direct sum decompositions into mutually orthogonal subspaces:
\beas
\mathcal L^{2}(\Omega_{1};\omega_{1})&= \bigoplus_{\chi\in \widehat{\G}_{\AAA}} \mathcal L^{2}_{\chi}(\Omega_{1};\omega_{1}),&&&
\mathcal A^{2}(\Omega_{1};\omega_{1}) &= \bigoplus_{\chi\in \widehat{\G}_{\AAA}} \mathcal A^{2}_{\chi}(\Omega_{1};\omega_{1}).
\eeas
\end{enumerate}
\end{theorem}
\vskip0.1in 
\noindent The proof of Theorem \ref{Thm1.1} is given in Section \ref{Sec5}. 

\subsection{Isomorphisms between Bergman spaces} 

 Let $\Omega_{1}, \Omega_{2}\subseteq\mathbb C^{n}$ be open sets, let $\mathbf A\in \mathbb M_{n}(\mathbb Z)$ be non-singular, and recall that $\mathbb H=\big\{(z_{1}, \ldots, z_{n})\in \mathbb C^{n}:\text{$z_{j}=0$ for some $1\leq j \leq n$}\big\}$.  Since $\mathbf A$ is non-singular, it is easy to check that $\z\in \mathbb H$ if and only if $\Phi_{\mathbf A}(\mathbf z)\in \mathbb H$. We shall suppose
\begin{equation} \label{Domains-Assumption-1} \Omega_2 = \Phi_{\mathbf A}(\Omega_1) \text{ and }   \Phi_{\mathbf A}: \Omega_1 \rightarrow \Omega_2  \text{ is holomorphic}
\end{equation}  
though not necessarily biholomorphic. In particular this means that if $\Omega_1 \cap \mathbb H \ne \emptyset$, then for every $j$ such that $\Omega_1 \cap \{ \mathbf z \in \mathbb C^n : z_j = 0 \}  \ne \emptyset$, the $j^{\text{th}}$ column of $\mathbf A$ has only non-negative integer entries. On the other hand if $\Omega_1 \cap \mathbb H = \emptyset$ then any non-singular $\mathbf A \in \mathbb M_n(\mathbb Z)$ generates a holomorphic map $\Phi_{\mathbf A}:\Omega_1\to\Omega_{2}$, and in this case $\Omega_2 \cap \mathbb H=\emptyset$ as well.  For $i=1,\,2$ set $\Omega_i^{\ast} := \Omega_i \setminus \mathbb H$.  It follows from \eqref{Domains-Assumption-1} that $\Omega_2^{\ast} = \Phi_{\mathbf A}(\Omega_1^{\ast})$. We shall assume that $\Omega_1^{\ast}$ is invariant under the action of $\mathbb G_{\mathbf A}$ defined in part \eqref{group-action} of Definition \ref{Def1.1}. Thus we assume 
\begin{equation} \label{Domains-Assumption-2}
\pmb{\xi}[\m] \otimes \z \in \Omega_1^{\ast} \text{ whenever } \z \in \Omega_1^{\ast}; \text{ or equivalently }  \Omega_1^{\ast} = \Phi_{\mathbf A}^{-1} (\Omega_2^{\ast}).
\end{equation} 

\vskip0.1in

\noindent Let $\omega_{j}:\Omega_{j}\to (0,\infty)$, $j=1,\,2$, be positive, continuous weight functions such that 
\begin{equation} \label{omega-12-relation}
\omega_{1}(\z)=\omega_{2}\big(\Phi_{\AAA}(\z)\big), \qquad \mathbf z \in \Omega_1. \end{equation} 
 In particular, this implies that the function $\omega_1$ is invariant  under the group action of $\mathbb G_{\mathbf A}$.   The measure $d \mu = \omega_1 dV$ is then also invariant under this action, in a sense that will be made precise in equation \eqref{invariant-measure} in Section \ref{Sec4}.  Under these conditions there is an isomorphism between $\mathcal A^2 (\Omega_1^{\ast}, \omega_1) $ and a direct sum of weighted Bergman spaces on $\Omega_{2}^{*}$.

\begin{theorem}\label{Thm1.2}
Let $\Omega_1$ and $\Omega_2$ be open sets in $\mathbb C^n$ satisfying assumptions \eqref{Domains-Assumption-1} and \eqref{Domains-Assumption-2}. Let $\omega_1$ and $\omega_2$ be continuous weight functions satisfying \eqref{omega-12-relation}.  Let $\b\in \Z^{n}$ and let $\chi=\varphi([\![\b]\!])$ be the  character of $\G_{\AAA}$ defined \eqref{Eqn1.17} so that $\chi([\m])=\exp\big[2\pi i \langle\m,\b\cdot\AAA^{-1}\rangle\big]$. Let $\Pi_{\chi}$ be the mapping defined in \eqref{Eqn1.12}.
\begin{enumerate}[(a)]
\smallskip
\item \label{Thm2.2a}
If $f:\Omega_{1}^{\ast} \to \C$ is \emph{any} function, there exists a unique function $T_{\b}[f]:\Omega_{2}^{\ast} \to \C$ so that $T_{\b}[f]\big(\Phi_{\AAA}(\z)\big)=\Pi_{\chi}[f](\z)\,F_{\b}(\z)$ for all $\z\in \Omega_{1}^{\ast}$.
\smallskip
\item \label{Thm2.2b}
If $g: \Omega_2^{\ast} \rightarrow \mathbb C$ is \emph{any} function and if $f(\z) = g \circ \Phi_{\mathbf A}(\mathbf z) F_{-\mathbf b}(\mathbf z)$, then $\Pi_{\chi}[f] = f$ and $T_{\mathbf b} [f] = g$.

\smallskip
\item \label{Thm2.2c}
If $f$ is holomorphic on $\Omega_1^{\ast}$ then $T_{\mathbf b}[f]$ is holomorphic on $\Omega_2^{\ast}$. 

\smallskip
\item \label{Thm2.2d}
Let  $\c= \c(\b) := (\1-\b)\cdot\AAA^{-1}-\1$ and let $\eta_{\b}(\w) :=\det(\AAA)^{-1}|F_{\c}(\w)|^{2}\omega_{2}(\w)$.
Then for every $f\in \mathcal L^{2}(\Omega_{1};\omega_{1})$ 
\begin{equation} \label{norm-identity} 
\int_{\Omega_{1}}|\Pi_{\chi}[f](\z)|^{2}\,\omega_{1}(\z)\,dV(\z)=
\int_{\Omega_{2}}|T_{\b}[f](\w)|^{2}\,\eta_{\b}(\w)\,dV(\w).
\end{equation} 

\medskip
\item \label{Thm2.2e}
If $\mathcal L^{2}_{\chi}(\Omega_{1};\omega_{1}) =\Pi_{\chi}\big[\mathcal L^{2}(\Omega_{1};\omega_{1})\big]$ and $\mathcal A^{2}_{\chi}(\Omega_{1};\omega_{1}) =\Pi_{\chi}\big[\mathcal A^{2}(\Omega_{1};\omega_{1})\big]$, the mappings
\begin{equation}  \label{two-isomorphisms}
T_{\mathbf b}: \mathcal L^2_{\chi}(\Omega_1, \omega_1) \rightarrow \mathcal L^2(\Omega_2, \eta_{\mathbf b}) \quad \text{ and } \quad T_{\mathbf b}: \mathcal A^2_{\chi}(\Omega_1^{\ast}, \omega_1) \rightarrow \mathcal A^2(\Omega_2^{\ast}, \eta_{\mathbf b}) 
\end{equation} 
are isometric isomorphisms of Hilbert spaces.

\medskip
\item \label{Thm2.2f} 
For each $\chi\in \widehat{\mathbb G}_{\mathbf A}$ choose $\mathbf b_{\chi}\in \mathbb Z^{n}$ with $\varphi( [\![\mathbf b_{\chi} ]\!]) = \chi$. Then there is an isomorphism 
\begin{equation} \label{Bergman-Omega1-Omega2} 
\mathcal A^2 (\Omega_1^{\ast}, \omega_1) \cong \bigoplus\nolimits_{\chi \in \widehat{\mathbb G}_{\mathbf A}} \mathcal A^2(\Omega_2^{\ast}, \eta_{\mathbf b_{\chi}})  
\end{equation} 
and an identity of Bergman kernels:  for $\z, \w \in \Omega_1^{\ast}$, 
\begin{equation} \label{Bergman-kernel-identity}  B_{\Omega_1^{\ast}}(\mathbf z, \mathbf w; \omega_1) = \sum_{\chi \in \widehat{\mathbb G}_{\mathbf A}} F_{-\mathbf b_{\chi}} \circ \Phi_{\mathbf A}(\mathbf z) B_{\Omega_2^{\ast}} \bigl( \Phi_{\mathbf A}(\mathbf z), \Phi_{\mathbf A}(\mathbf w) ; \eta_{\mathbf b_{\chi}}\bigr) \overline{F_{-\mathbf b_{\chi}} \circ \Phi_{\mathbf A}(\mathbf w) }.  \end{equation}  
In particular, 
\begin{equation} \label{Bergman-diagonal} 
B_{\Omega_1^{\ast}}(\mathbf z, \mathbf z; \omega_1) =  \sum_{\chi \in \widehat{\mathbb G}_{\mathbf A}} \bigl| F_{-\mathbf b_{\chi}} \circ \Phi_{\mathbf A}(\mathbf z) \bigr|^2 B_{\Omega_2^{\ast}} \bigl( \Phi_{\mathbf A}(\mathbf z), \Phi_{\mathbf A}(\mathbf z) ; \eta_{\mathbf b_{\chi}}\bigr). 
\end{equation} 
\end{enumerate}
\end{theorem}

\noindent Theorem \ref{Thm1.2} is proved in Section \ref{Sec5}.

\smallskip

\noindent Since $\Omega_1^{\ast} \subseteq \Omega_1$, the extremal characterization \eqref{Eqn1.3a} gives the inequality $B_{\Omega_1^{\ast}}(\z, \z; \omega_1) \geq  B_{\Omega_1}(\z, \z; \omega_1)$.  Combining this with \eqref{Bergman-diagonal}, we get    
\begin{corollary} \label{Cor1.3}
Under the same hypotheses as Theorem \ref{Thm1.2}, we have for $\z \in \Omega_1^{\ast}$, 
\begin{equation} \label{Eqn1.29}
B_{\Omega_1}(\mathbf z, \mathbf z; \omega_1) \leq  \sum_{\chi \in \widehat{\mathbb G}_{\mathbf A}} \bigl| F_{-\mathbf b_{\chi}} \circ \Phi_{\mathbf A}(\mathbf z) \bigr|^2 B_{\Omega_2^{\ast}} \bigl( \Phi_{\mathbf A}(\mathbf z), \Phi_{\mathbf A}(\mathbf z) ; \eta_{\mathbf b_{\chi}}\bigr). 
\end{equation} 
\end{corollary} 

\noindent {\em{Remarks: }} 
\begin{enumerate}[1.]
\item 
In part (\ref{Thm2.2f}) of Theorem \ref{Thm1.2}, the choice of $\b_{\chi} \in \mathbb Z^n$ such that $\varphi([\![\b_{\chi}]\!] )= \chi$ is not unique. Different choices lead to different choices of $\c(\b_{\chi})$ and $\eta_{\b_{\chi}}$ as given in part (\ref{Thm2.2d}), and hence lead to different spaces $\mathcal A^2(\Omega_2, \eta_{\b_{\chi}})$. Thus \eqref{Bergman-Omega1-Omega2} can be viewed as a family of decompositions for  $\mathcal A^2(\Omega_1^{\ast}, \omega_1)$ rather than a single one. 
\vskip0.1in
\item \noindent In Theorem \ref{Thm1.2}, it is important to note that the isomorphism between the two spaces $\mathcal L^2(\Omega_1, \omega_1)$ and $\mathcal L^2(\Omega_2, \eta_{\mathbf b})$ does {\em{not}} in general lead to an isomorphism between the corresponding Bergman spaces $\mathcal A^2(\Omega_i, \cdot)$, but does lead to an isomorphism of the Bergman spaces $\mathcal A^2(\Omega_i^{\ast}, \cdot)$ for the axes-deleted domains. Indeed the key point in part \eqref{Thm2.2e} of Theorem \ref{Thm1.2} is that the mapping $ T_{\mathbf b} :  \mathcal A_{\chi}^2(\Omega_1^{\ast}, \omega_1) \rightarrow \mathcal A^2(\Omega_2^{\ast}, \eta_{\mathbf b})$ is onto, whereas {\em{a priori }} the mapping 
$T_{\mathbf b} :  \mathcal A_{\chi}^2(\Omega_1, \omega_1) \rightarrow \mathcal A^2(\Omega_2^{\ast}, \eta_{\mathbf b})$  need not be onto.  For example, suppose that 
\[ \Omega_1 = \bigl\{\mathbf z = (z_1, z_2) \in \mathbb C^2: |z_1 z_2| < 1, |z_2| < 1 \bigr\} \text{ and } \Phi_{\mathbf A}(z_1, z_2) = (z_1z_2, z_2). \]
Then $\mathbb G_{\mathbf A}$ is trivial (hence so is $\widehat{\mathbb G}_{\mathbf A}$), and \[ \Omega_2 = \Phi_{\mathbf A}(\Omega_1) = \bigl\{\mathbf w = (w_1, w_2) : |w_1|<1, |w_2| < 1 \bigr\}\] is the unit polydisk in $\mathbb C^2$. Let us now choose the weight function $\omega_2(\mathbf w) = |w_2|^6$ and the holomorphic function $g(\mathbf w) = w_1/w_2^2$ on $\Omega_2 \setminus \mathbb H$.  Set $\mathbf b = \mathbf 0$, so that $\mathbf c = (0, -1)$, and $\eta_{\mathbf b}(\mathbf w) = |w_2|^4 dV(\mathbf w)$. We observe that $g \in \mathcal A^2(\Omega_2^{\ast}, \eta_{\b})$. However, $g$ does not lie in $T_{\b} \bigl(\mathcal A^2_{\chi}(\Omega_1, \omega_1) \bigr)$ where $\chi$ is the identity character. This is because any $f \in \mathcal A^2_{\chi}(\Omega_1, \omega_1)$ with $T_{\mathbf b}[f] = g$ must satisfy $f(\mathbf z) = z_1/z_2$ on $\Omega_1 \setminus \mathbb H$. Such a function $f$ does not admit a holomorphic extension to the origin.  
\end{enumerate} 

\subsection{Bergman kernel estimates} 
The Bergman kernel identity \eqref{Bergman-kernel-identity} involves the axes-deleted domains $\Omega_1^{\ast}$ and $\Omega_2^{\ast}$ rather than the original domains $\Omega_1$ and $\Omega_2$. Also the upper bound in Corollary \ref{Cor1.3} is not sharp in general. In this section we state a result that for certain choices of domain-weight pairs $(\Omega_1, \omega_1)$, an identity like \eqref{Bergman-kernel-identity} holds for $\Omega_{1}$ and $\Omega_{2}$, and the inequality in \eqref{Eqn1.29} is an equality. We begin by specifying the type of weights for which such results will hold. 

\begin{definition}
\noindent Let $\Omega\subseteq\mathbb C^{n}$ be open and $\omega:\Omega\to (0,\infty)$ a continuous weight function.
\begin{enumerate}[(a)]
\item $\omega$ is said to be of \emph{monomial type} if there exists $\pmb{\mu}=(\mu_{1}, \ldots, \mu_{n})\in \mathbb R^{n}$ and a continuous function $\vartheta:\Omega\to (0,\infty)$ such that
\begin{equation}  \label{def-weight} \begin{aligned} &\omega(\mathbf z) =  |F_{\pmb{\mu}}(\mathbf z)|^2 \vartheta(\mathbf z),  \text{ and }   
\inf \bigl\{ \vartheta(\mathbf z) : \z \in \Omega \bigr\} > 0. 
\end{aligned} 
\end{equation}  
 
\item We call a monomial-type weight function $\omega$ \emph{admissible} if
\begin{equation} \label{def-admissible-weight}
\mu_j < 1/2  \text{ for each index  $1 \leq j \leq n$ such that }  \Omega \cap \bigl\{ \mathbf z \in \mathbb C^n : z_j = 0 \bigr\} \ne \emptyset. 
\end{equation} 
For example, the weight function $\omega \equiv 1$ corresponding to the standard Bergman space is admissible. 
\end{enumerate}
\end{definition} 

\begin{proposition} \label{Bergman-star-prop} 
If $\omega: \Omega \rightarrow [0, \infty)$ is an admissible weight function of monomial type on $\Omega$, then $\mathcal A^2(\Omega, \omega) = \mathcal A^2(\Omega^{\ast}, \omega)$. 
\end{proposition}   
\noindent We then have the following Bergman kernel identities for $B_{\Omega_1}$ and $B_{\Omega_2}$. 
 
\begin{theorem} \label{Thm-Bergman-unweighted} 
Let $(\Omega_j, \omega_j)$, $j=1,2$ be as in Theorem \ref{Thm1.2}. 
\begin{enumerate}[(a)]
\item Suppose that $\mathcal A^2(\Omega_1, \omega_1) = \mathcal A^2(\Omega_1^{\ast}, \omega_1)$. Then the identities \eqref{Bergman-kernel-identity} and \eqref{Bergman-diagonal} hold, with $B_{\Omega_1^{\ast}}$ on the left side replaced by $B_{\Omega_1}$. In particular, this is the case whenever $\omega_1$ is admissible of monomial type and satisfies \eqref{omega-12-relation}.  \label{omega1-admissible} 

\vskip0.1in

\item Suppose that $\omega_2$ is a weight function of monomial type on $\Omega_2$, not necessarily admissible. Then for every $\chi \in \widehat{\mathbb G}_{\mathbf A}$, there exists a choice $\b_{\chi} \in \mathbb Z^n$ such that $\varphi([\![ \b_{\chi} ]\!]) = \chi$ and such that the weight function $\eta_{\b_{\chi}}$ is admissible of monomial type on $\Omega_2$. For such choices the identities \eqref{Bergman-kernel-identity} and \eqref{Bergman-diagonal} hold, with $B_{\Omega_2^{\ast}}$ on the right side of those relations replaced by $B_{\Omega_2}$. \label{omega2-admissible}
\vskip0.1in
\item Suppose that both $\omega_1$ and $\omega_2$ are weight functions of monomial type obeying \eqref{omega-12-relation}, and that $\omega_1$ is admissible. Then for each $\chi \in \widehat{\mathbb G}_{\mathbf A}$ there exist $\mathbf b_{\chi}\in \mathbb Z^{n}$ such that \label{Omega-12}
\begin{align} 
 B_{\Omega_1}(\mathbf z, \mathbf w; \omega_1) &= \sum_{\chi \in \widehat{\mathbb G}_{\mathbf A}} F_{-\mathbf b_{\chi}} \circ \Phi_{\mathbf A}(\mathbf z) B_{\Omega_2} \bigl( \Phi_{\mathbf A}(\mathbf z), \Phi_{\mathbf A}(\mathbf w) ; \eta_{\mathbf b_{\chi}}\bigr) \overline{F_{-\mathbf b_{\chi}} \circ \Phi_{\mathbf A}(\mathbf w) } \label{Bergman-identity-2} \\
 \label{Bergman-unweighted} 
B_{\Omega_1}(\mathbf z, \mathbf z; \omega_1) &=  \sum_{\chi \in \widehat{\mathbb G}_{\mathbf A}} \bigl| F_{-\mathbf b_{\chi}} \circ \Phi_{\mathbf A}(\mathbf z) \bigr|^2 B_{\Omega_2} \bigl( \Phi_{\mathbf A}(\mathbf z), \Phi_{\mathbf A}(\mathbf z) ; \eta_{\mathbf b_{\chi}}\bigr). 
\end{align} 
In particular, the relations \eqref{Bergman-identity-2} and \eqref{Bergman-unweighted} hold when $\omega_1 \equiv 1$, i.e., for the standard Bergman space on $\Omega_1$.  
\end{enumerate} 
\end{theorem}
\noindent Proposition \ref{Bergman-star-prop} and Theorem \ref{Thm-Bergman-unweighted} are proved in Section \ref{Sec5}.

\medskip

\noindent {\em{Remark: }} 
It is important to note the distinction between Theorem \ref{Thm1.2} (\ref{Thm2.2d}) and Theorem \ref{Thm-Bergman-unweighted} (\ref{omega2-admissible}) and (\ref{Omega-12}). The identities \eqref{Bergman-kernel-identity} and \eqref{Bergman-diagonal} hold for the axes-deleted domains $\Omega_1^{\ast}$ and $\Omega_2^{\ast}$ equipped with arbitrary continuous weight functions $\omega_1$ and $\omega_2$ obeying \eqref{omega-12-relation}, and these identities remain valid for {\em{any}} choice of $\b_{\chi} \in \mathbb Z^n$ obeying $\varphi([\![ \b_{\chi} ]\!] = \chi$. In contrast, the relations \eqref{Bergman-identity-2} and \eqref{Bergman-unweighted} are true for the original domains $\Omega_1$ and $\Omega_2$ and for certain choices of $\b_{\chi}$, provided the associated weights are of appropriate monomial type.

\subsection{A simple example}\quad

\vskip0.1in

\noindent Before developing the general theory we consider a very simple example of our main results. Let $\Omega_{1}=\Omega_{2}=\mathbb D=\big\{z\in\mathbb C:|z|<1\big\}$, and let $\Phi:\mathbb D\to\mathbb D$ be the proper mapping $\Phi(z)=z^{2}$. The standard Bergman kernel and projection for the unit disk are given by 
$$
B_{\mathbb D}(z,w)= \frac{1}{\pi}(1-z\overline w )^{-2}\quad\text{and}\quad \mathcal P_{\mathbb D}[f](z)=\frac{1}{\pi}\int_{\mathbb D}\frac{f(w)}{(1- z\overline w)^{2}}dV(w).
$$
It follows from Bell's work that for this example,  equations (\ref{transformation-law}) and (\ref{Eqn1.7})  become
\beas
\text{$f\in \mathcal L^{2}(\mathbb D)$ and $g(z)=2zf(z^{2})$}\Longrightarrow \mathcal P_{\mathbb D}[g](z)&= 2z\,\mathcal P_{\mathbb D}[f](z^{2}),\\
\frac{1}{2{\sqrt {\overline w}}}B_{\D}(z,\sqrt w) -\frac{1}{2{\sqrt {\overline w}}}B_{\D}(z,-\sqrt w)&=2z\,B_{\D}(z^{2},w).
\eeas 

\noindent Our approach is to decompose $h\in \mathcal A^{2}(\mathbb D)$ into even and odd functions, and then identify the corresponding subspaces of $\mathcal A^{2}(\mathbb D)$ with certain weighted Bergman spaces. If $h\in \mathcal A^{2}(\mathbb D)$ set
\begin{align*}
\Pi_{e}[h](z) &=\frac{1}{2}\big(h(z)+h(-z)\big),& \mathcal A^{2}_{e}(\mathbb D)&=\big\{h\in \mathcal A^{2}(\mathbb D):h(z)=h(-z)\big\},\\
 \Pi_{o}[h](z) &= \frac{1}{2}\big(h(z)-h(-z)\big), &  \mathcal A^{2}_{o}(\mathbb D)&=\big\{h\in \mathcal A^{2}(\mathbb D):h(z)=-h(-z)\big\}.
 \end{align*}
We see that $\mathcal A^{2}_{e}(\mathbb D)$ and  $\mathcal A^{2}_{o}(\mathbb D)$ are closed complementary orthogonal subspaces of $\mathcal A^{2}(\mathbb D)$, and hence $\mathcal A^{2}(\mathbb D) = \mathcal A^{2}_{e}(\mathbb D)\oplus \mathcal A^{2}_{o}(\mathbb D)$ with $||h||^{2}_{2} =||\Pi_{e}h||^{2}_{2}+||\Pi_{o}h||^{2}_{2}$ where$|| \cdot||_2$ denotes the norm in $\mathcal L^2(\mathbb D)$. Next if $h\in \mathcal  A^{2}(\mathbb D)$ there are unique holomorphic functions $\pi_{e}[h]$ and $\pi_{o}[h]$ on $\D$ so that $\Pi_{e}[h](z)=\pi_{e}[h](z^{2})$ and $\Pi_{o}[h](z)=z\,\pi_{o}[h](z^{2})$. Since \[ \int_{\mathbb D}f(w) \,dV(w)=2\int_{\mathbb D}f(z^{2})|z|^{2}\,dV(z), \] it follows that
\begin{align*}
||\Pi_{e}[h]||^{2}_{2}&=\frac{1}{2}\int_{\mathbb D}|\pi_{e}[h](z)|^{2}\,|z|^{-1}dV(z) \text{ and} \\ 
||\Pi_{o}[h]||^{2}_{2}&=\frac{1}{2}\int_{\mathbb D}|\pi_{o}[h](z)|^{2}\,dV(z).
\end{align*}
Thus if we introduce weight functions $\zeta_{e}(w)=\frac{1}{2}|w|^{-1}$ and $\zeta_{o}(w)\equiv \frac{1}{2}$ on $\D$,  the mappings 
\[ \pi_{e}: \mathcal A^{2}_{e}(\mathbb D)\to \mathcal A^{2}(\mathbb D;\zeta_{e}dV) \quad \text{ and } \quad \pi_{o}: \mathcal A^{2}_{o}(\mathbb D)\to \mathcal A^{2}(\mathbb D;\zeta_{o}dV) \] are isometric isomorphisms.  In particular if $\pi=(\pi_{e},\pi_{o})$
we have the following relation between  Bergman projections and Bergman kernels:
\bea\label{Eqn1.16}
\pi\circ \mathcal P_{\mathbb D}&= \big(\mathcal P_{\mathbb D}^{\zeta_{e}}\circ\pi_{e},\mathcal P_{\mathbb D}^{\zeta_{o}}\circ\pi_{o}\big),\\
B_{\mathbb D}(z,w)&= B_{\mathbb D}(z^{2},w^{2}; \zeta_e)+z\,\overline w\,B_{\mathbb D}(z^{2},w^{2}; \zeta_o).
\eea
In the notation of Theorems \ref{Thm1.1} and \ref{Thm1.2}, $\Phi_{\mathbf A}(z) = z^2$, $\mathbb G_{\mathbf A} \cong \widehat{\mathbb G}_{\mathbf A} \cong \{-1,0 \}$. Following the prescription of Theorem \ref{Thm1.2} \eqref{Thm2.2d}, we find that 
\[ c = \begin{cases} 0 &\text{ if } b = -1, \\ -\frac{1}{2} &\text{ if } b = 0, \end{cases} \quad \text{ and  hence } \quad \eta_b(z) = \begin{cases} \frac{1}{2} = \zeta_o &\text{ if } b = -1, \\ \frac{1}{2}|z|^{-1} = \zeta_e &\text{ if } b= 0.  \end{cases}   \]  
Thus equation \eqref{Bergman-identity-2} shows that 
$ B_{\mathbb D} (z, w) = z \overline{w} B_{\mathbb D} (z^2, w^2; \eta_{-1}) + B_{\mathbb D} (z^2, w^2; \eta_{0})$, which is \eqref{Eqn1.16}. 

\section{Functions and mappings of monomial type} \label{Sec-Monomial-Map}
\noindent We collect here basic facts concerning the functions and maps of the form \eqref{def-Phi}. Set 
\bea\label{Eqn2.1}
\mathbb O^{n}&:=\big\{(t_{1}, \ldots, t_{n})\in \mathbb R^{n}:\text{$t_{j}>0$  for $1 \leq j \leq n$}\big\},\\
\C^{n}_{*}&:=\big\{(z_{1}, \ldots, z_{n})\in \C^{n}:\prod_{j=1}^{n}z_{j}\neq 0\big\}=\C^{n}\setminus\H.
\eea
Thus $\O^{n}$ is the positive octant in $\R^{n}$ and $\C^{n}_{*}$ is $\C^{n}$ with complex coordinate planes deleted.  We denote by $\mathbf 1=(1, \ldots,1)\in \Z^{n}$ the vector with $1$ in every entry.  The vector $\e_{k}=(0, \ldots, 1,\ldots,0)\in \Z^{n}$ is the unit vector with $1$ in the $k^{th}$ entry and zeros elsewhere. If $\mathbf a = (a_1, \cdots, a_n) \in \mathbb R^n$ then $F_{\a}(\t) =F_{\mathbf a}(t_{1}, \ldots, t_{n})= t_{1}^{a_{1}}t_{2}^{a_{2}}\cdots t_{n}^{a_{n}}$ is a \emph{function of monomial-type} and $F_{\mathbf a}:\mathbb O^{n}\to\mathbb (0,\infty)$. If each $a_{j}\in \Z$ then  $F_{\mathbf a}$ extends to a holomorphic function on $\mathbb C^{n}_{*}$. If also each $a_{j}\geq 0$ then $F_{\mathbf a}$ extends to a holomorphic function on $\mathbb C^{n}$.  For $\{\a_{1}, \ldots, \a_{n} \} \subset \R^{n}$, let $\AAA\in \M_{n}(\R)$  be the matrix whose $j^{th}$ row vector is $\mathbf a_j$. Then $\Phi_{\mathbf A}(\mathbf t) = (F_{\mathbf a_1}(\mathbf t), \cdots, F_{\mathbf a_n}(\mathbf t))$ is \emph{mapping of monomial-type} corresponding to $\mathbf A$, and $\Phi_{\AAA}:\O^{n}\to\O^{n}$. If $\mathbf A\in \M_{n}(\Z)$ then $\Phi_{\mathbf A}$ is a holomorphic mapping from $\mathbb C^n_{*} $ to itself. If all the entries of $\AAA$ are non-negative then $\Phi_{\mathbf A}$ is a holomorphic mapping from $\mathbb C^n$ to itself.  Let $J\Phi_{\AAA}(\t)=\det\left(\frac{\partial F_{\a_{j}}}{\partial t_{k}}\right)(\t)$ denote the Jacobian matrix. For the proof of the following, see Lemma 4.2 in  \cite{NP-2009}.

\goodbreak

\begin{proposition} \label{Prop2.1}
Let $\t\in \O^{n}$. 
\quad
\begin{enumerate}[(a)]
\item \label{Prop2.1a}
If $\b_{j}\in\R^{n}$, $c_{j}\in \R$, and $\a=\sum_{j=1}^{k}c_{j}\b_{j}$ then $F_{\a}(\t)=\prod_{j=1}^{k}F_{\b_{j}}(\t)^{c_{j}}$;

\smallskip

\item \label{Prop2.1b}
If $\AAA, \BBB\in \M_{n}(\R)$ and $\a\in\R^{n}$ then $F_{\a\cdot\AAA}(\t)=F_{\a}\big(\Phi_{\AAA}(\t) \big)$ and $\Phi_{\AAA\cdot\BBB}(\t)=\Phi_{\AAA}\big(\Phi_{\BBB}(\t)\big)$;

\smallskip

\item \label{Prop2.1c}
Let  $\AAA\in \M_{n}(\R)$  and  $\b=\1\cdot\AAA-\1\in \R^{n}$. Then $J\Phi_{\AAA}(\t)= \det(\AAA)\,F_{\b}(\t)$.

\smallskip

\item \label{Prop2.1d}
If $\AAA\in \M_{n}(\R)$ is invertible then $\Phi_{\AAA}:\O^{n}\to\O^{n}$ is a diffeomorphism and  $\Phi_{\AAA}^{-1}=\Phi_{\AAA^{-1}}$. 
\smallskip
\end{enumerate}
\end{proposition}

\noindent The identities in (\ref{Prop2.1a}), (\ref{Prop2.1b}), and (\ref{Prop2.1c}) continue to hold for $\t=\z\in \C^{n}_{*}$, and also for $\t=\z\in \C^{n}$ provided the vectors and matrices have non-negative integer entries.  If $\AAA\in \M_{n}(\Z)$ is invertible and $|\det(\AAA)|\neq 1$ then $\Phi_{\AAA}:\C^{n}_{*}\to\C^{n}_{*}$ is not one-to-one, and so is not biholomorphic.  However, we have the following replacement.

\goodbreak

\begin{proposition}\label{Prop2.2} Let $\AAA\in \M_{n}(\Z)$ be non-singular. 

\begin{enumerate}[(a)]

\smallskip

\item\label{Prop2.2a}
$\Phi_{\AAA}:\C^{n}_{*}\to\C^{n}_{*}$ is a proper holomorphic mapping.  

\medskip

\item \label{Prop2.2c}
If $\Phi_{\AAA}(\z)=\w\in\C^{n}_{*}$ there is a neighbourhood $U_{\z}$ of $\z$ in $\C^{n}_{*}$ so that $\Phi_{\AAA}:U_{\z}\to\Phi_{\AAA}(U_{\z})$ is a biholomorphic mapping.

\end{enumerate}
\end{proposition}
\begin{proof}
If $\AAA \in \mathbb M_{n}(\Z)$ is invertible, then the rows $\a_{j}$ of $\AAA$ form a basis for $\R^{n}$. Solving the linear system $\mathbf A \mathbf x = \e_k$ using  Cramer's rule, we have $\e_{k}=\det(\AAA)^{-1}\sum_{j=1}^{n}b_{j,k}\a_{j}$ where each $b_{j,k}\in\Z$. It follows from part (\ref{Prop2.1a}) of Proposition \ref{Prop2.1} that
\bea\label{Eqn2.5}
z_{k}^{\det(\AAA)}=F_{\e_{k}}(\z)^{\det(\AAA)}=\prod_{j=1}^{n}F_{\a_{j}}(\z)^{b_{j,k}}.
\eea
Suppose now that $K$ is a compact subset of $\C^{n}_{*}$. For part \eqref{Prop2.2a}, we need to show that $\Phi_{\mathbf A}^{-1}(K) = \{ \mathbf z \in \mathbb C_{\ast}^n : \Phi_{\mathbf A}(\mathbf z) \in K \} \subset \mathbb C_{\ast}^n$ is compact, i.e., closed and bounded. That the latter set is closed follows easily from the fact that $K$ is closed and $\Phi_{\mathbf A}$ is continuous. To prove that $\Phi_{\mathbf A}^{-1}(K)$ is bounded, we observe that there exist positive numbers $\epsilon < N$ such that $K\subset\big\{\w \in \C^{n}:\epsilon\leq |w_{k}|\leq N,\,1\leq k \leq n\big\}$. Thus, if $\mathbf w = \Phi_{\AAA}(\z)=\big(F_{\a_{1}}(\z), \ldots, F_{\a_{n}(\z)}\big)\in K$ then we have that $\epsilon\leq |F_{\a_{j}}(\z)|\leq N$ for $1 \leq j \leq n$. It follows from (\ref{Eqn2.5}) that each $|z_{k}|$ is bounded and bounded away from zero by constants depending on $\epsilon$, $N$, the integers $b_{j,k}$, and $\det(\AAA)$. This implies that $\Phi_{\mathbf A}^{-1}(K)$  is compact, proving (\ref{Prop2.2a}). Part (\ref{Prop2.2c}) follows from the holomorphic inverse function theorem since $J\Phi_{\AAA}(\z)\neq 0$ for all $\z\in \C^{n}_{*}$.
\end{proof}

\vskip0.1in
\noindent The explicit nature of $\Phi_{\mathbf A}$ allows us to describe the pre-image of any point in $\mathbb C_{\ast}^n$. To this end, and for any $\z\in \C^{n}_{*}$ , let us write its polar form \[ \z =\big(r_{1}e^{2\pi i \theta_{1}}, \ldots, r_{n}e^{2\pi i \theta_{n}}\big) = \mathbf r \otimes \exp[2 \pi i \pmb{\theta}], \] where $\r=(r_{1}, \ldots, r_{n})\in \O^{n}$ and $\exp[\mathbf v] = (e^{v_1}, \cdots, e^{v_n})$ for any row vector $\mathbf v$. Thus $\mathbf r$ is uniquely determined and $\thetab=(\theta_{1}, \ldots, \theta_{n})\in \R^{n}$ is determined up to translation by an element of $\Z^{n}$. 

\begin{lemma}\label{Lem2.3}
Let $\AAA\in \M_{n}(\Z)$ be non-singular and let $\w=\rhob \otimes \exp[2\pi i \phib]\in \C^{n}_{*}$.

\begin{enumerate}[(a)]

\item \label{Lem2.3a}Then there exists $\z_{0}\in \C^{n}_{*}$ such that $\Phi_{\AAA}(\z_{0})=\w$.

\smallskip

\item \label{Lem2.3b}If $\z_1, \z_2 \in \mathbb C_{\ast}^n$ are given by the polar forms  $\z_{1}=\r_{1} \otimes \exp [2\pi i \thetab_{1}]$ and $\z_{2}=\r_{2} \otimes \exp[2\pi i \thetab_{2}]$,  then $\Phi_{\AAA}(\z_{1})=\Phi_{\AAA}(\z_{2})$ if and only if $\r_{1}=\r_{2}$ and $(\thetab_{1}-\thetab_{2})\cdot\AAA^{t}\in \Z^{n}$.

\smallskip
\item \label{Lem2.3c}The inverse image of $\w$ under $\Phi_{\AAA}$ can be identified with the group $\G_{\AAA}$ via the one-to-one and onto mapping $\G_{\AAA}\ni [\m]\longrightarrow \xib([\mathbf m]) \otimes \z_{0}$, defined in \eqref{Eqn1.11}. In particular, for every $\mathbf w \in \mathbb C_{\ast}^n$, the cardinality of $\Phi_{\mathbf A}^{-1}(\mathbf w)$ is the same, and equals $\#(\mathbb G_{\mathbf A})$.
\end{enumerate}
\end{lemma}
\begin{proof}
If $\z=\r\,e^{2\pi i \thetab}$ then $\Phi_{\AAA}(\z)= \big(F_{\a_{1}}(\r)e^{2\pi i \langle\a_{1},\thetab\rangle}, \ldots, F_{\a_{n}}(\r)e^{2\pi i \langle\a_{n},\thetab\rangle}\big)$, and so  
\beas
\Phi_{\AAA}(\z)&=\w &&\Longleftrightarrow &
 &\big(F_{\a_{1}}(\r)e^{2\pi i \langle\a_{1},\thetab\rangle}, \ldots, F_{\a_{n}}(\r)e^{2\pi i \langle\a_{n},\thetab\rangle}\big)=\big(\rho_{1}e^{2\pi i \phi_{1}}, \ldots, \rho_{n}e^{2\pi i \phi_{n}}\big)\\
 &&&\Longleftrightarrow &
 &\Phi_{\AAA}(\r)=\rhob \quad\text{and}\quad \thetab\cdot\AAA^{t}=\phib+\m\qquad\text{for some $\m\in \Z^{n}$}\\
 &&&\Longleftrightarrow &
 & \Phi_{\AAA}(\r)=\rhob \quad\text{and}\quad \thetab=\phib\cdot(\AAA^{-1})^{t}+\m\cdot(\AAA^{-1})^{t}\quad\text{for some $\m\in \Z^{n}$}.
\eeas
Now $\Phi_{\AAA}:\O^{n}\to\O^{n}$ is invertible by Proposition \ref{Prop2.1}, part (\ref{Prop2.1d}). For the rest of this proof, let us denote by $\Phi_{\mathbf A}^{-1}(\pmb{\rho})$ the unique pre-image of $\pmb{\rho}$ in $\mathbb O^n$. Set $\z_{0}=\Phi_{\AAA}^{-1}(\rhob) \otimes \exp\big[2\pi i \phib\cdot(\AAA^{-1})^{t}\big]$. It follows that $\Phi_{\AAA}(\z_{0})=\w$, proving (\ref{Lem2.3a}). Next, if $\Phi_{\AAA}(\r_{1} \otimes \exp[2\pi i \thetab_{1}])=\Phi_{\AAA}(\r_{2} \otimes \exp[2\pi i \thetab_{2}])$ then $\Phi_{\AAA}(\r_{1})=\Phi_{\AAA}(\r_{2})$ and $(\thetab_{1}-\thetab_{2})\cdot \AAA^{t}=\m$ for some $\m\in \Z^{n}$. Since $\Phi_{\AAA}$ is invertible on $\O^{n}$ it follows that $\r_{1}=\r_{2}$, proving (\ref{Lem2.3b}). Finally, for part \eqref{Lem2.3c} the computation above shows that the inverse image of $\w$ under $\Phi_{\AAA}$ is contained in the set of all points of the form 
\beas
\z_{\m}=\Phi_{\AAA}^{-1}(\rhob) \otimes \exp\big[2\pi i \phib\cdot(\AAA^{-1})^{t}+ 2 \pi i \m\cdot(\AAA^{-1})^{t}\big],\quad\m\in \Z^{n}.
\eeas 
But $\z_{\m_{1}}=\z_{\m_{2}}$ if and only if $(\m_{1}-\m_{2})\cdot(\AAA^{-1})^{t}\in \Z^{n}$, which means that $\m_{1}-\m_{2}\in \mathfrak C(\AAA)$. Thus the map $\m \in \mathbb Z^n \mapsto \z_{\m}$ lifts naturally to $[\m] \in \mathbb G_{\mathbf A} \mapsto \z_{\m} = \z_{[\m]}$. Moreover, a comparison of the expression above with the definition of $\xib$ in \eqref{Eqn1.11} shows that 
\[ \z_{[\m]} = \mathbf z_0 \otimes \exp \bigl[ 2 \pi i \m \cdot (\mathbf A^t)^{-1}  \bigr] = \z_0 \otimes \xib([\m]). \] 
This completes the proof.
 \end{proof}

\section{The group $\G_{\AAA}$}\label{Sec3}

\noindent To describe the structure of the group $\G_{\AAA}$ we use a normal form for integer matrices, sometimes called the \emph{Smith normal form}.  We state this in the lemma below; a proof of it can be found in  \cite[Chapter 12, Theorem 4.3]{Artin}, \cite[Chapter 1, Theorem 11.3]{Munkres}, or \cite[Chapter 3, Theorem 5]{Jacobson}.

\begin{lemma}\label{Lem3.1}
Let  $\AAA\in \M_{n}(\Z)$ be non-singular.
\begin{enumerate}[(a)]

\item \label{Lem3.1a}
There exist $\SSS, \TTT, \mathbf \Lambda \in \M_{n}(\Z)$ with $|\det(\SSS)|=|\det(\TTT)|=1$ and $\mathbf \Lambda$ a diagonal matrix such that 
\begin{equation} \label{Smith-normal} \SSS\cdot\AAA\cdot\TTT=\mathbf \Lambda. \end{equation} 

\smallskip

\item \label{Lem3.1b}
The diagonal entries $\lambda_{1}, \ldots, \lambda_{n}$ of $\mathbf\Lambda$  satisfy $1\leq \lambda_{1}\leq \lambda_{2}\leq \cdots\leq \lambda_{n}$, and each $\lambda_{j}$ divides $\lambda_{j+1}$ for $1 \leq j \leq n-1$. They are called the \emph{invariant factors} of $\AAA$. 
\end{enumerate}
\end{lemma}

\begin{remark}\label{Rem3.2}
Since $|\det(\mathbf S)|=|\det(\mathbf T)|=1$, it follows from Cramer's rule that the inverse matrices $\SSS^{-1},\,\TTT^{-1}\in \M_{n}(\Z)$; \emph{i.e.} they also have integer entries. 
\end{remark}

\subsection{Structure of $\G_{\AAA}$}\label{Sec3.1}

\smallskip

\noindent If $\lambda_{1}, \ldots, \lambda_{n}$ are the invariant factors of a non-singular matrix $\AAA\in \M_{n}(\Z)$, the next lemma shows that  $\G_{\AAA}$ is isomorphic to $\bigoplus_{j=1}^{n} \Z/\lambda_{j} \Z$. In order to define the isomorphism, we set \begin{equation} s = \max \{ \ell: 1 \leq \ell \leq n, \; \lambda_{\ell} = 1\}, \end{equation}  with the convention that $s = 0$ if $\lambda_1 > 1$.  Then 
\begin{equation}
\bigoplus_{j=1}^{n} \Z/\lambda_{j} \Z=\bigoplus_{j=s+1}^{n}\Z/\lambda_{j}\Z. \label{group-identification} 
\end{equation} 
 Let $\SSS\in \M_{n}(\Z)$ be the matrix from Lemma \ref{Lem3.1}. Then define $\widehat\iota:\Z^{n}\to \bigoplus_{j=1}^{n} \Z/\lambda_{j}\Z$ by 
 \begin{equation} \label{def-hat-i} 
 \widehat\iota(\m)=\big(\pi_{1}(\langle\e_{1},\m\cdot\SSS^{t}\rangle), \ldots, \pi_{n}(\langle\e_{n},\m\cdot\SSS^{t}\rangle\big), \end{equation} 
 where $\pi_{j}:\Z\to\Z/\lambda_{j}\Z$ is the group homomorphism which sends each integer $m\in \Z$ to its equivalence class $[m]_{\lambda_{j}}\in \Z/\lambda_{j}\Z$. If $s \geq 1$, then $\pi_j \equiv 0$ for $j \leq s$. It is easy to see  that  $\widehat\iota$ is a group homomorphism.

\begin{lemma} \label{Lem3.3}
For $\widehat{\iota}$ as in \eqref{def-hat-i}, the following conclusions hold.  
\begin{enumerate}[(a)]

\item  \label{Lem3.3a}
$\widehat\iota\big(\mathfrak C(\AAA)\big)=\0$; hence  $\widehat \iota$ induces a group homomorphism $\iota:\G_{\AAA}\to \bigoplus_{j=1}^{n} \Z/\lambda_{j}\Z$.

\smallskip

\item \label{Lem3.3b}
The homomorphism \,$\iota$\,is an isomorphism and so $ \#(\G_{\AAA})=|\prod_{j=1}^{n}\lambda_{j}|=|\det(\AAA)|$.

\smallskip

\item \label{Lem3.3d} $\G_{\AAA}$ is generated  by the $(n-s)$ elements $\big\{\big[\e_{s+1}\cdot(\SSS^{t})^{-1}\big], \ldots, \big[\e_{n}\cdot(\SSS^{t})^{-1}\big]\big\}$.  
\end{enumerate}
\end{lemma} 
\begin{proof} 
If $\n\in \mathfrak C(\AAA)$ then $\n=\m\cdot\AAA^{t}$ for some $\m\in \Z^{n}$. It follows from the Smith normal form \eqref{Smith-normal} for $\mathbf A$ that \[ \n\cdot \SSS^{t}=\m\cdot\AAA^{t}\cdot \SSS^{t}=\m\cdot (\TTT^{-1})^{t}\cdot\mathbf \Lambda\cdot (\SSS^{-1})^{t}\cdot\SSS^{t}=\m\cdot(\TTT^{-1})^{t}\cdot\mathbf\Lambda. \] Since $\m\cdot(\TTT^{-1})^{t}\in \Z^{n}$ it follows that the $k^{th}$ entry of $\n\cdot\SSS^{t}$ is an integer multiple of $\lambda_{k}$ and so $\pi_k(\langle\e_{k},\n\cdot\SSS^{t})=[0]_{\lambda_{k}}$, and this establishes (\ref{Lem3.3a}).
\vskip0.1in
\noindent Now suppose that $\n\in \Z^{n}$ and that $\iota([\mathbf n]) = \mathbf 0$. Then $\pi_{j}(\langle\e_{j},\n\cdot\SSS^{t}\rangle)=[0]_{\lambda_{j}}\in \Z/\lambda_{j}\Z$ so there exists $\mathbf k = (k_1, \cdots, k_n) \in \mathbb Z^n$ such that $\mathbf n\cdot \mathbf S^t = (k_1 \lambda_{1}, \cdots, k_n \lambda_n) = \mathbf k \cdot \mathbf \Lambda$. Thus 
\begin{equation*} 
\n=\k\cdot\mathbf\Lambda\cdot (\SSS^{t})^{-1}=\k\cdot\mathbf\Lambda^{t}\cdot (\SSS^{t})^{-1}=\k\cdot\TTT^{t}\cdot\AAA^{t}\in\mathfrak C(\AAA), 
\end{equation*}
which means that $[\n]=[\mathbf 0]\in \G_{\AAA}$, and so $\iota$ is one-to-one. Next, fix any $([k_{1}]_{\lambda_{1}}, \cdots, [k_n]_{\lambda_{n}}) \in \bigoplus_{\ell=1}^{n} \mathbb Z/\lambda_{\ell} \mathbb Z$. Set $\mathbf k = (k_{1}, \cdots, k_n)$ and $\mathbf n = \mathbf k \cdot(\mathbf S^t)^{-1} \in \mathbb Z^n$. Then 
\begin{align*} 
\iota([\n])= \widehat\iota(\n)&=\big(\pi_{1}(\langle\e_{1},\mathbf k \cdot(\mathbf S^t)^{-1}\cdot\SSS^{t}\rangle), \ldots, \pi_{n}(\langle\e_{n},\mathbf k \cdot(\mathbf S^t)^{-1}\cdot\SSS^{t}\rangle)\big) \\
&= \big(\pi_{1}(\langle\e_{1},\mathbf k \rangle), \ldots, \pi_{n}(\langle\e_{n},\mathbf k \rangle)\big) \\ &=
([k_{1}]_{\lambda_{1}}, \cdots, [k_n]_{\lambda_{n}}),  
\end{align*} 
which shows that $\iota$ is surjective, proving  (\ref{Lem3.3b}). 
\vskip0.1in
\noindent It remains to prove part \eqref{Lem3.3d}. For $\lambda_j > 1$, the group $\Z/\lambda_{j}\Z$ under addition is generated by $[1]_{\lambda_{j}}$. In view of the identification \eqref{group-identification}, the group $\bigoplus_{j=1}^{n} \Z/\lambda_{j}\Z$ is generated by the elements $\{\mathbf f_k : s+1 \leq k \leq n \}$, with $\mathbf f_k = \bigl(0, \cdots, 0, [0]_{\lambda_{s+1}}, \ldots, [1]_{\lambda_{k}}, \ldots, [0]_{\lambda_{n}} \bigr)$, where all the components are zero except the $k^{th}$, which is one. But $\iota\big(\big[\e_{k}\cdot(\SSS^{t})^{-1}\big]\big)= \mathbf f_k$ since
\beas[]
\big[\e_{j}\cdot (\SSS^{t})^{-1}\big]_{\lambda_{k}} =\pi_{k}\big(\langle\e_{k},\e_{j}\cdot(\SSS^{t})^{-1}\cdot\SSS^{t}\rangle\big)=\pi_{k}\big(\langle\e_{k},\e_{j}\rangle\big)
=\begin{cases}
[1]_{\lambda_{j}}&\text{if $k=j$}\\
[0]_{\lambda_{j}}&\text{if $k\neq j$}
\end{cases},
\eeas
completing the proof. 
\end{proof} 

\begin{corollary}\label{Cor3.4}
If $\AAA\in \M_{n}(\Z)$ is non-singular and $\w\in\C^{n}_{*}$ then the cardinality of the inverse image $\Phi_{\AAA}^{-1}(\{\w\})$ is $|\det(\AAA)|$.
\end{corollary}
\begin{proof} 
This follows by combining Lemma \ref{Lem2.3} \eqref{Lem2.3c} and Lemma \ref{Lem3.3} \eqref{Lem3.3b}.  
\end{proof} 

\subsection{The characters of $\G_{\AAA}$}\label{Sec3.2}\quad

\smallskip
\noindent A \emph{character} of a group $G$ is a group homomorphism $\chi:G\to \T=\{z\in \C:|z|=1\}$. The set of all characters is denoted by $\widehat G$ and is a group under point-wise multiplication. Let us recall that if $\m\in \Z^{n}$, then $[\m]$ denotes its equivalence class in $\G_{\AAA}=\Z^{n}/\mathfrak C(\AAA)$ and $[\![\m ]\!]$ denotes its equivalence class in $\G_{\AAA^{t}}=\Z^{n}/\mathfrak C(\AAA^{t})$.

\begin{lemma}\label{Lem3.5}Let $\AAA\in \M_{n}(\Z)$ be non-singular.
\begin{enumerate}[(a)]

\item\label{Lem3.5a}
If $\chi$ is a character of $\G_{\AAA}$   there exists $\b\in \Z^{n}$ so that 
\begin{equation} 
\chi([\m])= \exp \bigl[2\pi i \langle\m,\b\cdot\AAA^{-1}\rangle \bigr] \quad \text{ for every $[\m]\in \G_{\AAA}$}.
\end{equation}
 Conversely, every $\b\in\Z^{n}$ defines a character in this way.

\medskip

\item \label{Lem3.5b}Two elements $\b_{1}, \b_{2}\in \Z^{n}$ define the same character $\chi$ if and only if $\b_{1}-\b_{2}= \n\cdot\AAA$ for some $\n\in \Z^{n}$; \emph{i.e.} $\b_{1}-\b_{2}\in \mathfrak C(\AAA^{t})$.

\medskip

\item \label{Lem3.5c} If  \,$[\![\b]\!]\in \G_{\AAA^{t}}$ define $\varphi([\![\b]\!])([\m]) :=\exp\big[2\pi i \langle\m,\b\cdot\AAA^{-1}\rangle\big]$. Then $\varphi:\G_{\AAA^{t}}\to\widehat{\G}_{\AAA}$ is a group isomorphism.

\end{enumerate}
\end{lemma}
\begin{proof}

It is easy to check that every character $\widehat{\chi}$ of \,$\Z^{n}$ is given by 
\begin{equation*} 
\widehat{\chi}(\m)=\exp\big[2\pi i \langle\m,\thetab\rangle\big]  \text{ for a unique } \thetab=(\theta_{1}, \ldots, \theta_{n})\in\T^{n}.
\end{equation*} 
 If $\chi$ is a character of $\G_{\AAA}=\Z^{n}/\mathfrak C(\AAA)$, it lifts to a character $\widehat \chi$ of $\Z^{n}$ and so $\widehat\chi(\m)=\exp \bigl[ 2\pi i \langle\m,\thetab_{\chi}\rangle \bigr]$ for a unique $\thetab_{\chi}\in\T^{n}$. Moreover $\widehat\chi$ must be the identity on $\mathfrak C(\AAA)$, hence 
 \begin{equation} 
 \exp \bigl[ 2\pi i \langle\m\cdot\AAA^{t},\thetab_{\chi}\rangle \bigr] =1  \text{ for all }  \m\in \Z^{n}. 
 \end{equation} 
 Thus $\thetab_{\chi}\cdot \AAA\in \Z^{n}$ and so $\thetab_{\chi}= \b\cdot \AAA^{-1}$ for some $\b\in \Z^{n}$. On the other hand, if $\b\in \Z^{n}$ then $\m\longrightarrow  \exp \bigl[ 2\pi i \langle \m,\b\cdot\AAA^{-1}\rangle \bigr]$ is clearly a character of $\Z^{n}$ which is the identity on the subgroup $\m\in \mathfrak C(\AAA)$.  This proves (\ref{Lem3.5a}). 
 \vskip0.1in
\noindent If $\b_{1}, \b_{2}\in \Z^{n}$ define the same character of $\G_{\AAA}$ then $\exp \bigl[ 2\pi i \langle\m,(\b_{1}-\b_{2})\cdot \AAA^{-1}\rangle \bigr]=1$ for all $\m\in \Z^{n}$, and so $(\b_{1}-\b_{2})\cdot \AAA^{-1}=\p\in \Z^{n}$, and this proves (\ref{Lem3.5b}). Assertion (\ref{Lem3.5c}) then follows from (\ref{Lem3.5a}) and (\ref{Lem3.5b}).
\end{proof}
\goodbreak
\noindent We now define certain special characters of $\G_{\AAA}$. Let $\varphi:\G_{\AAA^{t}}\to\widehat{\G}_{\AAA}$ be the isomorphism from part (\ref{Lem3.5c}) of Lemma \ref{Lem3.5}. For $1 \leq k \leq n$ define
\begin{equation}
\xi_{k}([\m]) =\varphi([\![\e_{k}]\!])([\m]) = \exp\big[2\pi i \langle\m,\e_{k}\cdot\AAA^{-1}\rangle\big]\in\T. \label{Eqn3.1}
\end{equation}
We note that this agrees with the definition of $\xi_k$ in \eqref{Eqn1.10}.

\begin{proposition}\label{Prop3.6}
The characters $\xi_{1}, \ldots, \xi_{n}$ generate the dual group $\widehat{\G}_{\AAA}$.
\end{proposition}

\begin{proof}
By part \eqref{Lem3.5a} of Lemma \ref{Lem3.5}, every character $\chi$ of $\G_{\AAA}$ is of the form $\chi=\varphi([\![\b]\!])$ for some $\b\in \Z^{n}$.
We can write $\b=\sum_{k=1}^{n}b_{j}\e_{k}\in \Z^{n}$, from which it follows that \[ \chi=\varphi([\![\b]\!])=\prod_{k=1}^{n}\varphi([\![\e_{k}]\!])^{b_{k}}=\prod_{k=1}^{n}\xi_{k}^{b_{k}}. \]
The last equation shows that $\{\xi_1, \cdots, \xi_n \}$ is a set of generators. 
\end{proof}

\subsection{The action of $\G_{\AAA}$}\label{Sec3.3} Given the characters $\xi_{k}$ defined in (\ref{Eqn3.1}) let us recall the definition of $\xib$ in \eqref{Eqn1.10}. We use $\xib$ to define an action of the group $\G_{\AAA}$ on the set $\C^{n}$ via the relation \eqref{group-action}. A group action is \emph{faithful} if for any two distinct elements of the group, there exist some element of the set that produces distinct images under the action.

\begin{lemma}\label{Lem3.7}
The following conclusions hold. 
\begin{enumerate}[(a)]
\item \label{Lem3.7a} For any $[\m_1], [\m_2] \in \mathbb G_{\mathbb A}$ and any $\mathbf z \in \mathbb C_{\ast}^n$, 
\begin{enumerate}[(i)]

\smallskip

\item \label{Lem3.7ai}$\xib([\m_{1}+\m_{2}])\otimes \z=\xib([\m_{1}])\otimes\big(\xib([\m_{2}])\otimes \z\big)$;

\smallskip

\item if $\xib([\m_{1}])\otimes\z=\xib([\m_{2}])\otimes\z$ for some $\z\in \C^{n}_{*}$ then $[\m_{1}]=[\m_{2}]$.
\end{enumerate}
\smallskip
In particular, the mapping $([\m],\z)\to \xib([\m])\otimes \z$ is a faithful action of $\G_{\AAA}$ on $\C^{n}_{*}$.
\medskip

\item \label{Lem3.7b} $\Phi_{\AAA}\big(\xib([\m])\otimes\z\big)=\Phi_{\AAA}(\z)$ for all $[\m]\in \G_{\AAA}$ and all $\z\in \C^{n}_{*}$. 

\medskip

\item \label{Lem3.7c}If $\z_{1},\z_{2}\in \C^{n}_{*}$ and $\Phi_{\AAA}(\z_{1})=\Phi_{\AAA}(\z_{2})$ then there exists $[\m]\in\G_{\AAA}$ such that $\xib[\m]\otimes\z_{2}=\z_{1}$.

\medskip

\item \label{Lem3.7d}
If $\b\in \Z^{n}$ then $F_{\b}\big(\xib([\m])\big)=\exp\big[2\pi i \langle\m,\b\cdot\AAA^{-1}\rangle\big]$.
\medskip
\end{enumerate}
\end{lemma}

\begin{proof} 
The property in (i) of part (\ref{Lem3.7a}) follows from the fact that each $\xi_{k}$ is a character of $\G_{\AAA}$. If $\xib[\m_{1}]\otimes\z=\xib[\m_{2}]\otimes\z$ for some $\z\in\C^{n}_{*}$ then $\xib([\m_{1}])=\xib([\m_{2}])$ since all the components of $\z$ are non-zero. Hence  for $1 \leq k \leq n$, 
\begin{align*}  
&\exp\big[2\pi i \big\langle\m_{1}, \e_{k}\cdot\AAA^{-1}\big\rangle\big] =\exp\big[2\pi i \big\langle\m_{2}, \e_{k}\cdot\AAA^{-1}\big\rangle\big], \text{ or } \\ &\exp \bigl[ 2 \pi i \big\langle \e_k, (\m_1 - \m_2) (\mathbf A^t)^{-1} \big \rangle \bigr] = 1, 
\end{align*}  
so $(\m_{1}-\m_{2})\cdot (\AAA^{t})^{-1}\in \Z^{n}$. But this says that $\m_{1}-\m_{2}\in \mathfrak C(\AAA)$ and so $[\m_{1}]=[\m_{2}]$, completing the proof of (\ref{Lem3.7a}). 
\vskip0.1in
\noindent Part \eqref{Lem3.7b} follows from Lemma \ref{Lem2.3} \eqref{Lem2.3c}, with $\z_0 = \z$ and  $\mathbf w = \Phi_{\mathbf A}(\mathbf z)$ in the notation of that lemma.   
\vskip0.1in
\noindent Next, let $\z_{\ell}=\r_{\ell} \otimes \exp[2\pi i \thetab_{\ell}]\in \C^{n}_{*}$ for $\ell =1,\,2$. From Lemma \ref{Lem2.3} \eqref{Lem2.3b} we know that \[ \Phi_{\AAA}(\z_{1})=\Phi_{\AAA}(\z_{2})  \text{ if and only if }  \r_{1}=\r_{2}  \text{ and } \thetab_{2}=\thetab_{1}+\m\cdot(\AAA^{t})^{-1} \] for some $\m\in \Z^{n}$. But this is true if and only if  $\z_{2}=\xib([\m])\otimes \z_{1}$, establishing (\ref{Lem3.7c}).  
\vskip0.1in
\noindent Finally, for part \eqref{Lem3.7d} we have 
\begin{align*}  F_{\b}\big(\xib([\m])\big) &=\prod_{j=1}^{n} \exp \bigl[ {2\pi i b_{j}\langle\m,\e_{j}\cdot\AAA^{-1}\rangle} \bigr] \\ &=  \exp \Bigl[ {2\pi i \big\langle\m,\sum_{j=1}^{n} b_j \e_{j}\cdot\AAA^{-1} \big \rangle} \Bigr]= \exp \bigl[ {2\pi i \big \langle \m,\b\cdot\AAA^{-1} \big \rangle} \bigr]. \qedhere 
\end{align*} 
\end{proof}

\section{Group actions and characters}\label{Sec4}
\noindent In this section we recall some basic facts about group characters and group actions for an arbitrary finite abelian group $G$.   Let $\#(G)$ denote the cardinality of $G$ and let $e$ denote the identity element. The following orthogonality relations are well-known; see for example \cite[Corollary 4.2]{Conrad} for a proof. 

\begin{proposition} \label{Prop4.1}
Let $G$ be a finite abelian group. Suppose that $\chi_{1}, \chi_{2}\in \widehat G$ and $g\in G$. Then
\beas
\sum_{g\in G}\chi_{1}(g)\,\overline{\chi_{2}(g)} &= 
\begin{cases}
\#(G) &\text{if $\chi_{1}=\chi_{2}$},\\
0 &\text{if $\chi_{1}\neq\chi_{2}$}
\end{cases} &&\text{and}&
\sum_{\chi\in\widehat G}\chi(g) &= 
\begin{cases}
\#(G) &\text{if $g=e$ },\\
0 &\text{if $g\neq e$}.
\end{cases}
\eeas
\end{proposition}

\noindent Now suppose that $G$  has an action $\rho$ on a set $X$; i.e. $\rho$ is a group homomorphism from $G$ to the group of permutations of $X$. If $g\in G$ the action of $\rho(g)$ on $x\in X$ is denoted by $\rho(g)\cdot x$. The action $\rho$ is \emph{faithful}  if $\rho(g_{1})=\rho(g_{2})$ implies $g_{1}=g_{2}$.  If $F:X\to\C$ and  $\chi\in \widehat G$ set
\bea\label{Eqn3.3q}
\Pi_{\chi}[F](x) :=F_{\chi}(x) = \frac{1}{\#(G)} \sum_{g\in G}\chi(g)\,F(\rho(g)\cdot x).
\eea
Note that in the context of this paper, the above definition of $\Pi_{\chi}$ is the same as the one specified in \eqref{Eqn1.12}.  A positive measure $\mu$ on $X$ is said to be {\em{invariant under the action $\rho$ of $G$}} if for all $g\in G$ and all $f\in L^{1}(X,d\mu)$ the following relation holds: \begin{equation}  \int_{X}f(x)\,d\mu(x)=\int_{X}f(\rho(g)\cdot x)\,d\mu(x). \label{invariant-measure} 
\end{equation}

\begin{proposition}\label{Prop4.2} 
Suppose that $G$ is a finite abelian group, equipped with an action $\rho$ on a set $X$, which in turn supports a positive measure $\mu$ that is invariant under $\rho$. Let $F:X\to \C$.

\begin{enumerate}[(a)]
\item \label{Prop4.2a}
$F(x)=\sum_{\chi\in\widehat G}F_{\chi}(x)$.

\smallskip

\item \label{Prop4.2b}
If $h\in G$ then $F_{\chi}(\rho(h)\cdot x)= \overline{\chi(h)}\,F_{\chi}(x)=\chi(h)^{-1}F_{\chi}(x)$.

\smallskip

\item \label{Prop4.2c}
If $G:X\to \C$ and $G(\rho(h)\cdot x)=\chi(h)^{-1}G(x)$ for all $h\in G$ then $\Pi_{\chi}[G]=G$.

\smallskip

\item \label{Prop4.2d}
If $F\in L^{2}(X,d\mu)$ and $\chi_{1},\chi_{2}\in \widehat G$ then 
\beas
\int_{X}F_{\chi_{1}}(x)\overline{F_{\chi_{2}}(x)}\,d\mu(x) =
\begin{cases}
||F_{\chi_{1}}||^{2}_{\mathcal L^{2}(X,d\mu)} &\text{if $\chi_{1}=\chi_{2}$,}\\
0 &\text{if $\chi_{1}\neq \chi_{2}$.}
\end{cases}
\eeas

\item \label{Prop4.2e}
If $\chi_{1}, \chi_{2}\in \widehat G$ then \,\, $\Pi_{\chi_{1}}\big[\Pi_{\chi_{2}}[F]\big](x) = 
\begin{cases}
\Pi_{\chi_{1}}[F](x) &\text{if $\chi_{1}=\chi_{2}$,}\\
0&\text{if $\chi_{1}\neq\chi_{2}$.}
\end{cases}$
\end{enumerate}
\end{proposition}

\begin{proof}
We have $\sum_{\chi\in\widehat G}F_{\chi}(x)= \frac{1}{\#(G)}\sum_{g\in G} \Big[\sum_{\chi\in\widehat G}\chi(g)\Big]\,F(\rho(g)\cdot x)= F(x)$  by Proposition \ref{Prop4.1}, which gives (\ref{Prop4.2a}). Next, using a reparametrization $g \in G \mapsto gh$ in the sum, we arrive at the relation 
\beas
F_{\chi}(\rho(h)\cdot x)&=\#(G)^{-1}\sum_{g\in G}\chi(g)\,F\big(\rho(g)\cdot\rho(h)\cdot x)
=
\#(G)^{-1}\sum_{g\in G}\chi(g)\,F\big(\rho(gh)\cdot x)\\
&=
\#(G)^{-1}\sum_{g\in G}\chi(gh^{-1})\,F\big(\rho(g)\cdot x)
=
\overline{\chi(h)}\,\#(G)^{-1}\sum_{g\in G}\chi(g)\,F\big(\rho(g)\cdot x),
\eeas
which gives (\ref{Prop4.2b}).  A similar argument gives (\ref{Prop4.2c}). If $\chi_{1}\neq \chi_{2}$ then the defining property \eqref{invariant-measure} of an invariant measure yields  
\beas
\int_{X}F_{\chi_{1}}(x)\overline{F_{\chi_{2}}(x)}\,d\mu(x)
&=
\#(G)^{-2}\int_{X}\sum_{g,h\in G}\chi_{1}(g) \overline{\chi_{2}(h)}F(\rho(g)\cdot x)\overline{F(\rho(h)\cdot x)}\,d\mu(x)\\
&=
\#(G)^{-2}\int_{X}\sum_{g,h\in G}\chi_{1}(g) \overline{\chi_{2}(h)}F(\rho(g h^{-1})\cdot x)\overline{F( x)}\,d\mu(x)\\
&=
\#(G)^{-2}\int_{X}\sum_{g,h\in G}\chi_{1}(gh) \overline{\chi_{2}(h)}F(\rho(g )\cdot x)\overline{F(x)}\,d\mu(x)
\\
&=
\#(G)^{-2}\int_{X}\sum_{g\in G}\chi_{1}(g)\Big[\sum_{h\in G}\chi_{1}(h) \overline{\chi_{2}(h)}\Big]F(\rho(g )\cdot x)\overline{F( x)}\,d\mu(x)=0
\eeas
by Proposition \ref{Prop4.1}. Finally 

\beas
\Pi_{\chi_{1}}\big[\Pi_{\chi_{2}}[F]\big](x) &= \Pi_{\chi_{1}}\Big[\#(G)^{-1}\sum_{g\in G}\chi_{2}(g) F(\rho(g)\cdot x)\Big]\\
&=\#(G)^{-2}\sum_{g\in G}\chi_{2}(g)\sum_{h\in G}\chi_{1}(h)\,F(\rho(g)\cdot\rho(h)\cdot x)\\
&=
\#(G)^{-2}\sum_{g,h\in G}\chi_{1}(h)\chi_{2}(g)\,F(\rho(gh)\cdot x)\\
&=
\#(G)^{-2}\sum_{g,h\in G}\chi_{1}(h)\chi_{2}(gh^{-1})\,F(\rho(g)\cdot x)\\
&=
\#(G)^{-2}\sum_{g\in G}\chi_{2}(g)\Big[\sum_{h\in G}\chi_{1}(h)\overline{\chi_{2}(h)}\Big]\,F(\rho(g)\cdot x),
\eeas
and (\ref{Prop4.2d}) then follows from Proposition \ref{Prop4.1}.
\end{proof}

\begin{corollary}\label{Cor4.3}
If $F\in \mathcal L^{2}(X,d\mu)$, then the norm of $F$ decomposes as follows: \[ ||F||^{2}_{\mathcal L^{2}(X,d\mu)}=\sum_{\chi \in \widehat G} ||F_{\chi}||^{2}_{\mathcal L^{2}(X,d\mu)}. \] 
\end{corollary}

\medskip

\noindent Now let $\mathcal S(X)$ be a vector space of functions on $X$ that is invariant under the action $\rho$ of $G$, i.e., if $F \in \mathcal S(X)$, then for every $g \in G$, the function given by $x \in X \mapsto F(\rho(g) \cdot)$ is also in $\mathcal S(X)$. In this case, it follows from the definition \eqref{Eqn3.3q} of $\Pi_{\chi}$ and Proposition \ref{Prop4.1} that the linear operator $\Pi_{\chi}$ maps $\mathcal S(X)$ into $\mathcal S(X)$. For each $\chi\in \widehat G$ set
\begin{equation} \label{def-Schi} 
\mathcal S(X)_{\chi} := \Pi_{\chi} \bigl[ \mathcal S(X) \bigr] = \big\{F\in \mathcal S(X) : F=\Pi_{\chi}[F]\big\}\subset \mathcal S(X).
\end{equation} 
The following is then an easy consequence of Proposition \ref{Prop4.2}.

\begin{corollary}\label{Cor4.4}
For each $\chi\in \widehat G$, the mapping $\Pi_{\chi}: \mathcal S(X)\to \mathcal S(X)$ is a linear map. Moreover,
\begin{enumerate}[(a)]
\medskip
\item $\mathcal S(X)_{\chi_{1}}\cap \mathcal S(X)_{\chi_{2}}=\{0\}$ if $\chi_{1}\neq \chi_{2}$.
\medskip
\item $\mathcal S(X) = \bigoplus_{\chi\in \widehat G} \mathcal S(X)_{\chi}$. \label{Cor4.4b} 
\medskip
\item If $\mathcal S(X)\subset \mathcal L^{2}(X,d\mu)$ then the subspaces $\{\mathcal S(X)_{\chi} : \chi \in \widehat{G} \}$ are mutually orthogonal.
\end{enumerate}
\medskip
In particular, if $\mathcal S(X)\subset \mathcal L^{2}(X,d\mu)$, then $\Pi_{\chi}$ acting on $\mathcal S(X)$ is an orthogonal projection, in which case the direct sum in \eqref{Cor4.4b} gives an orthogonal decomposition of $\mathcal S(X)$ into mutually orthogonal subspaces. 
\end{corollary}

\section{Proofs of theorems \ref{Thm1.1}, \ref{Thm1.2}, and \ref{Thm-Bergman-unweighted}}\label{Sec5}

\subsection{Proof of Theorem \ref{Thm1.1}} \quad 
\medskip

\noindent According to Lemma \ref{Lem3.7}  the mapping $([\m],\z)\to \xib([\m ])\otimes\z$ from equation \eqref{Eqn1.11} defines a group action $\rho$ of $\G_{\AAA}$ on $\C^{n}_{*}$. Set $X = \Omega_1 \cap \mathbb C_{\ast}^n = \Omega_1 \setminus \mathbb H$. By assumption, $\Omega_{1}$ is invariant under this action; hence it follows from the definition of domain invariance on page \pageref{domain-invariance}  that $([\m],\z)\to \xib([\m ])\otimes\z$ generates a group action of $\mathbb G_{\mathbf A}$ on $X$ as well. Further, if $\omega_1$ is a continuous non-negative weight function that is invariant under $\rho$ according to the definition on page \pageref{function-invariance}, then the measure $d\mu(\mathbf z) = \omega_1(\mathbf z) dV(\mathbf z)$ is also an invariant measure on $X$, in the sense of \eqref{invariant-measure}. 
\vskip0.1in
\noindent We now choose two $\rho$-invariant vector spaces of functions. The first is $\mathcal S(X) = \mathcal L^2(\Omega_1 \setminus \mathbb H, d\mu)$, which is isomorphic to $\mathcal L^2(\Omega_1, \omega_1)$, since $\mathbb H$ has Lebesgue measure zero. The second choice of $\mathcal S(X)$ is the subspace of $\mathcal L^2(\Omega_1, \omega_1)$ consisting of holomorphic functions on $\Omega_1 \setminus \mathbb H$ that admit a holomorphic extension to $\Omega_1 \cap \mathbb H$. Note that the latter space is isomorphic to the weighted Bergman space $\mathcal A^2(\Omega_1, \omega_1)$. The results of Section \ref{Sec4} apply for these choices of $\mathcal S(X)$. Parts  (\ref{Thm2.1d}) and  (\ref{Thm2.1e}) of Theorem \ref{Thm1.1} then follow immediately from Corollary \ref{Cor4.4}.

\medskip

\subsection {Proof of Theorem \ref{Thm1.2}} \quad

\medskip

\noindent 
\begin{proof}[{\em{\textbf {Part (\ref{Thm2.2a}):}}}] 
The following two identities follow respectively from part (\ref{Lem3.7d}) of Lemma \ref{Lem3.7} and part (\ref{Prop4.2b}) of Proposition \ref{Prop4.2}: for every $[\m] \in \mathbb G_{\mathbf A}$, 
\[ F_{\b}(\xib([\m])\otimes\z) =\chi([\m])\,F_{\b}(\z) \quad \text{ and } \quad \Pi_{\chi}[f](\xib([\m])\otimes\z)=\chi([\m])^{-1}\Pi_{\chi}[f](\z).\]  
Combining these two observations we find that 
\beas
F_{\b}(\xib([\m])\otimes\z)\Pi_{\chi}[f](\xib([\m])\otimes\z)=F_{\b}(\z)\,\Pi_{\chi}[f](\z),
\eeas 
so $\Pi_{\chi}[g]\,F_{\b}$ is invariant under the action of $\G_{\AAA}$. Thus, the function $f_{\b} = T_{\mathbf b}[f]$ defined by \begin{equation}  f_{\b}(\w)=F_{\b}(\z)\Pi_{\chi}[f](\z) \text{ for any } \z\in \Phi_{\AAA}^{-1}(\w), \; \mathbf w \in \Omega_2^{\ast} \label{def-Tb} \end{equation}  is well-defined as a function on $\Omega_2^{\ast}$.  
\end{proof}

\noindent \begin{proof}[{\em\textbf{Part (\ref{Thm2.2c}):}}] If $f$ is holomorphic on $\Omega_1^{\ast}$, then so is $\Pi_{\chi}[f]$. Any function $F_{\b}$ is holomorphic on $\mathbb C_{\ast}^n$, and hence on $\Omega_1^{\ast}$. The function $f_{\b}$ is a composition of their product $F_{\b}(\cdot) \Pi_{\chi}[f](\cdot)$ with $\Phi_{\mathbf A}^{-1}$. We have shown in Proposition \ref{Prop2.2}(\ref{Prop2.2c}) that $\Phi_{\AAA}: \mathbb C_{\ast}^n \rightarrow \mathbb C_{\ast}^n$ is locally a biholomorphic mapping.  Thus $f_{\b}$ is holomorphic on $\Omega_2^{\ast}$.  
\end{proof}

\noindent 
\begin{proof}[\em{\textbf{Part (\ref{Thm2.2b}):}}] Next, for any $g: \Omega_2^{\ast} \rightarrow \mathbb C$, let us set $f(\z)=g \big(\Phi_{\AAA}(\z)\big)\,F_{-\b}(\z)$. We know from Lemma \ref{Lem3.7}(\ref{Lem3.7c}) and (\ref{Lem3.7d}) that $\Phi_{\AAA}(\xib([\m])\otimes\z)=\Phi_{\AAA}(\z)$ and $F_{-\b}\big(\xib([\m]) \big)=\chi([\m])^{-1}$. Substituting these into the expression \eqref{Eqn1.12} for $\Pi_{\chi}[f]$, we arrive at 
\begin{align*}
\Pi_{\chi}[f](\z)&=  \frac{1}{\#(\mathbb G_{\mathbf A})} \sum_{[\m]\in \mathbb G_{\mathbf A}} \chi([\m]) f (\xib[\m] \otimes \z) \\
&= \frac{1}{\#(\G_{\AAA})}\sum_{[\m]\in \G_{\AAA}}\chi \bigl([\m])g \circ \Phi_{\AAA}(\xib([\m])\otimes\z \bigr) F_{-\b}\big(\xib([\m])\otimes\z\big)\\
%&=
%\#(\G_{\AAA})^{-1}\sum_{[\m]\in \G_{\AAA}}\chi([\m])g_{\b}\big(\Phi_{\AAA}(\z)\big)F_{-\b}\big(\xib([\m])\big)F_{-\b}(\z)\\
&=  \frac{1}{\#(\G_{\AAA})}\sum_{[\m]\in \G_{\AAA}}\chi \bigl([\m])g \big(\Phi_{\AAA}(\z)\big) \chi([\m])^{-1} F_{-\b}(\z) = f(\z).
\end{align*} 
The same relation also yields the second claim; namely, for every $\w = \Phi_{\mathbf A}(\z) \in \Omega_2$, \[T_{\b}[f](\w) = \Pi_{\chi}[f](\z) F_{\b}(\z) = f \circ \Phi_{\mathbf A}(\z) F_{\b}(\z) = g(\w). \qedhere\] 
\end{proof}

\noindent 
\begin{proof}[{\em{\textbf {Part (\ref{Thm2.2d}):}}}] For any integrable function $G:\Omega_{2}\to \C$,  the change of variables $\w=\Phi_{\AAA}(\z)$ gives 
\begin{equation}  \int_{\Omega_{2}}G(\w)\,dV(\w)=\det(\AAA)^{-1}\int_{\Omega_{1}}G\big(\Phi_{\AAA}(\z)\big)|J\Phi_{\AAA}(\z)|^{2}\,dV(\z), \label{change-of-variable}  \end{equation}   since  $\Phi_{\AAA}$ is $\det(\AAA)$-to-one. Now $\omega_{2}\big(\Phi_{\AAA}(\z)\big)=\omega_{1}(\z)$, and 
\begin{equation}  |J\Phi_{\AAA}(\z)|^{2}=\det(\AAA)^{2}|F_{\1\cdot\AAA-\1}(\z)|^{2}, \quad  |F_{\c}\big(\Phi_{\AAA}(\z)\big)|^{2}=|F_{\c\cdot\AAA}(\z)|^{2}. \label{cov-2} 
\end{equation}  
Given any $f \in \mathcal L^2(\Omega_1, \omega_1)$, there exists according to part (\ref{Thm2.2a}) a unique function $f_{\b} = T_{\b}[f]:\Omega_{2}\to \C$ so that
$g_{\b}\big(\Phi_{\AAA}(\z)\big)=\Pi_{\chi}[f](\z)\,F_{\b}(\z)$. Applying the change of variable formula \eqref{change-of-variable} with $ G(\w) = |T_{\b}[f](\w)|^2 \eta_{\b}(\w)$ and invoking the relations \eqref{cov-2}, we obtain 
\begin{align*}
\int_{\Omega_{2}}|T_{\b}[f](\w)|^{2} &\eta_{\b}(\mathbf w) \,dV(\w)
= \det(\mathbf A)^{-1} \int_{\Omega_{2}}|T_{\b}[f](\w)|^{2}|F_{\c}(\w)|^{2}\omega_{2}(\w)\,dV(\w)\\
& = 
\det(\AAA)^{-2}\int_{\Omega_{1}}|\Pi_{\chi}[f](\z)|^{2}|F_{\b}(\z)|^{2}|F_{\c}\big(\Phi_{\AAA}(\z)\big)|^{2}|J\Phi_{\AAA}(\z)|^{2}\omega_{2}(\Phi_{\AAA}(\z))\,dV(\z)\\
&=
\int_{\Omega_{1}}|\Pi_{\chi}[f](\z)|^{2}|F_{\c\cdot\AAA+\b}(\z)|^{2}|F_{\1\cdot\AAA-\1}(\z)|^{2}\omega_{1}(\z)\,dV(\z)\\
&=
\int_{\Omega_{1}}|\Pi_{\chi}[f](\z)|^{2}|F_{\c\cdot\AAA+\b+\1\cdot\AAA-\1}(\z)|^{2}\omega_{1}(\z)\,dV(\z)\\
&=
\int_{\Omega_{1}}|\Pi_{\chi}[f](\z)|^{2}\omega_{1}(\z)\,dV(\z).
\end{align*} 
 This completes the proof of \eqref{Thm2.2d}. 
 \end{proof}

\noindent 
\begin{proof}[{\em{\textbf{Part (\ref{Thm2.2e}):}}}] The definition \eqref{def-Tb} of $T_{\mathbf b}$ shows that it is linear.  Let us first show that $T_{\mathbf b}$ is injective, on $\mathcal L^2_{\chi}(\Omega_1, \omega_1)$, and hence on $\mathcal A^2_{\chi}(\Omega_1^{\ast}, \omega_1)$. If $T_{\b}[f] \equiv 0$ in $\mathcal L^2_{\chi}(\Omega_1, \omega_1)$, then the definition \eqref{def-Schi} dictates that $\Pi_{\chi}[f] = f$. It follows then from the norm identity \eqref{norm-identity} in part (\ref{Thm2.2d}) that 
\[ T_{\b}[f] \equiv 0 \text{ on } \mathcal L^2(\Omega_2, \eta_{\b}) \implies \Pi_{\chi}[f] = f \equiv  0 \text{ on } \mathcal L^2(\Omega_1, \omega_1), \text{ proving injectivity.}  \] 
Surjectivity of $T_{\b}$ on $\mathcal L^2(\Omega_2, \omega_2)$ and on $\mathcal A^2(\Omega_2^{\ast}, \eta_{\b})$ follows from parts \eqref{Thm2.2b} and \eqref{Thm2.2d}. Given $g \in \mathcal L^2(\Omega_2, \eta_{\b})$, the function $f$ defined in part \eqref{Thm2.2b} lies in $\mathcal L^2_{\chi}(\Omega_1, \omega_1)$ and obeys $T_{\mathbf b}[f] = g$. If $g \in \mathcal A^2(\Omega_2^{\ast}, \eta_{\b})$, the same function $f$ is also in $\mathcal A^2_{\chi}(\Omega_1^{\ast}, \omega_1)$, by virtue of part (\ref{Thm2.2c}). The identity \eqref{norm-identity} in part \eqref{Thm2.2d} shows that $T_{\b}$ preserves norms, and hence is an isometry. This proves the two isomorphisms claimed in \eqref{two-isomorphisms}.  
\end{proof}

\noindent 
\begin{proof}[{\em{\textbf{Part (\ref{Thm2.2f}):}}}] Since $\Omega_1^{\ast}$ and $\omega_1$ are both invariant under  the group action \eqref{group-action}, the space $\mathcal A^2(\Omega_1^{\ast}, \omega_1)$ admits the direct sum decomposition ensured by Theorem \ref{Thm1.1}. Combining this with part \eqref{Thm2.2e}, we see that 
\begin{align*}  \mathcal A^2(\Omega_1^{\ast}, \omega_1) &= \bigoplus \bigl\{ \mathcal A^2_{\chi}(\Omega_1^{\ast}, \omega_1) : \chi \in \widehat{\mathbb G}_{\mathbf A} \bigr\} \\
&\cong \bigoplus \bigl\{ \mathcal A^2_{\chi}(\Omega_2^{\ast}, \eta_{\b_{\chi}}) : \chi \in \widehat{\mathbb G}_{\mathbf A}, \varphi([\![\b_{\chi}]\!] =  \chi \bigr\}. \end{align*} 
\vskip0.1in
\noindent It remains to the prove the Bergman kernel identities \eqref{Bergman-Omega1-Omega2} and \eqref{Bergman-diagonal}.  For each $\chi\in \widehat{\G}_{\AAA}$ let $\big\{\psi^{\chi}_{1}(\cdot ; \omega_1), \psi^{\chi}_{2}(\cdot ; \omega_1), \ldots\big\}$  be a complete orthonormal basis for $\mathcal A^{2}_{\chi}(\Omega_{1}^{\ast},\omega_{1})$. Since \[ \mathcal A^{2}(\Omega_{1}^{\ast},\omega_{1})=\bigoplus\nolimits_{\chi\in \widehat{\G}_{\AAA}} \mathcal A^{2}_{\chi}(\Omega_{1}^{\ast},\omega_{1}) \]  is an orthogonal decomposition, the set of functions $\big\{\psi^{\chi}_{k}(\cdot ; \omega_1) :\chi\in \widehat{\G}_{\AAA},\,k\geq 1\big\}$ is a complete orthonormal basis for $\mathcal A^{2}(\Omega_{1}^{\ast}, \omega_{1})$. In view of \eqref{Eqn1.2}, this yields the following expression for the Bergman kernel:
\[ B_{\Omega_{1}^{\ast}}(\z,\w; \omega_1) = \sum_{\chi\in\widehat{\G}_{\AAA}}\sum_{k=1}^{\infty}\psi^{\chi}_{k}(\z; \omega_1)\overline{\psi^{\chi}_{k}(\w ; \omega_1)}, \qquad \z, \w \in \Omega_1^{\ast}.  \] 
For each $\chi \in \widehat{\mathbb G}_{\mathbf A}$, let us choose $\b_{\chi}\in\Z^{n}$ so that $\varphi([\![\b_{\chi}]\!])=\chi$. Then according to \eqref{Thm2.2e}, there are functions $\theta^{\chi}_{k}\in \mathcal A^{2}(\Omega_{2}^{\ast} , \eta_{\b_{\chi}})$, $k = 1, \,2,\ldots$, so that $T_{\b_{\chi}}[\psi^{\chi}_k (\cdot; \omega_1)]  = \theta^{\chi}_{k}(\cdot; \eta_{\b_{\chi}})$; in other words,  
\begin{equation} \label{theta-psi}  \theta^{\chi}_{k}\big(\Phi_{\AAA}(\z); \eta_{\b_{\chi}}\big)=\psi^{\chi}_{k}(\z; \omega_1)F_{\b_{\chi}}(\z). \end{equation}   
Since $T_{\b}$ is norm-preserving, and hence angle-preserving, for all $\b \in \mathbb Z^n$, it follows that for every $\chi \in \widehat{\mathbb G}_{\mathbf A}$, the set of functions $\{\theta_k^{\chi}(\cdot; \eta_{\b_{\chi}}) : k \geq 1 \}$ is a complete orthonormal basis for $\mathcal A^{2}(\Omega_{2}^{\ast} ;\eta_{\b_{\chi}})$. By \eqref{Eqn1.2}, this again  implies  
\begin{equation} \label{Bergman-Omega2}  B_{\Omega_{2}^{\ast}}(\xib,\zetab; \eta_{\b_{\chi}})=\sum_{k=1}^{\infty}\theta^{\chi}_{k}(\xib; \eta_{\b_{\chi}})\,\overline{\theta^{\chi}_{k}(\zetab; \eta_{\b_{\chi}})}, \qquad \xib, \zetab \in \Omega_2^{\ast}.  
\end{equation}  
Combining \eqref{theta-psi} and \eqref{Bergman-Omega2} gives  
\begin{align*}
B_{\Omega_{1}^{\ast}}(\z,\w; \omega_1) &= \sum_{\chi\in\widehat{\G}_{\AAA}}\sum_{k=1}^{\infty}\psi^{\chi}_{k}(\z; \omega_1)\overline{\psi^{\chi}_{k}(\w; \omega_1)}\\
&=
\sum_{\chi\in \widehat{\G}_{\AAA}}F_{-\b_{\chi}}(\z)\Big[\sum_{k=1}^{\infty}\big(\psi^{\chi}_{k}(\z)F_{\b_{\chi}}(\z)\big)\overline{\big(\psi^{\chi}_{k}(\w)\,F_{\b_{\chi}}(\w)\big)}\Big]\,\overline{F_{-\b_{\chi}}(\w)}\\
&=
\sum_{\chi\in \widehat{\G}_{\AAA}}F_{-\b_{\chi}}(\z)\Big[\sum_{k=1}^{\infty}\theta^{\chi}_{k}\big(\Phi_{\AAA}(\z)\big)\overline{\theta^{\chi}_{k}\big(\Phi_{\AAA}(\w)\big)}\Big]\,\overline{F_{-\b_{\chi}}(\w)}\\
&=
\sum_{\chi\in \widehat{\G}_{\AAA}}F_{-\b_{\chi}}(\z)\,B^{\etab_{\chi}}_{\Omega_{2}^{\ast}}\big(\Phi_{\AAA}(\z), \Phi_{\AAA}(\w)\big)\,\overline{F_{-\b_{\chi}}(\w)}.
\end{align*}
This is the relation \eqref{Bergman-Omega1-Omega2}. The relation \eqref{Bergman-diagonal}  follows from \eqref{Bergman-Omega1-Omega2} by setting $\z = \w$, completing the proof.  
\end{proof}

\subsection{Proof of Proposition \ref{Bergman-star-prop}} 
\begin{proof} 
\noindent Since the inclusion $\mathcal A^2(\Omega, \omega) \subseteq \mathcal A^2(\Omega^{\ast}, \omega)$ holds trivially for any domain $\Omega$ and any weight $\omega$, we will focus on the reverse inclusion. In other words, given any admissible weight $\omega$ of monomial type on $\Omega$, and any $f \in \mathcal A^2(\Omega^{\ast}, \omega)$, our goal is to show that $f$ admits a holomorphic extension to $\Omega \cap \mathbb H$. 
\vskip0.1in
\noindent Suppose that $\mathbf u = (u_1, \cdots, u_n) \in \Omega \cap \mathbb H$. Let $m$ denote the number of indices $1 \leq j \leq n$ such that $u_j=0$. Without loss of generality suppose that $u_1 = u_2 = \cdots = u_m = 0$. Write $\u = (\u', \u'')$, where $\u' = (u_1, \cdots, u_m) = \mathbf 0' \in \mathbb C^m$ and $\u'' = (u_{m+1}, \cdots, u_n) \in \mathbb C^{n-m}$. Choose $\epsilon > 0$ so that the polydisk \[ \mathbb D_{n}(\u; \epsilon \mathbf 1) = \bigl\{\z = (z_1, \cdots, z_n) : |z_j- u_j| < \epsilon, \text{ for }  1 \leq j \leq n  \bigr\} \]  is contained in $\Omega$.  In particular, choosing $\epsilon < \min \{|u_j|/2 : m < j \leq n \}$ ensures that  $\mathbb D_{n-m} (\u''; \epsilon \mathbf 1'')$ avoids the coordinate hyperplanes in $\mathbb C^{n-m}$. 
\vskip0.1in
\noindent Let $\mathbb A_m(\mathbf 0', \epsilon \mathbf 1')$ denote the $m$-fold Cartesian product of the punctured disk $\{z \in \mathbb C : 0 < |z| < \epsilon \}$. The choice of $\epsilon$ shows that 
 \[ \mathbb U_n = \mathbb A_m (\mathbf 0'; \epsilon \mathbf 1') \times \mathbb D_{n-m}(\u''; \epsilon \mathbf 1'')   \subseteq \Omega^{\ast}. \] 
Since $f \in \mathcal A^2(\Omega^{\ast}, \omega)$, it restricts to a function in $\mathcal A^2(\mathbb U_n, \omega)$. Let us recall that $\omega(\z) = \bigl| F_{\pmb{\mu}}(\z) \bigr|^2 \vartheta(\z)$ is of monomial type on $\Omega$. In view of the definition \eqref{def-weight}, this implies that $\vartheta$ is bounded above and below by positive constants on $\mathbb D_n(\mathbf 0; \epsilon \mathbf 1)$. Hence $\mathcal A^2(\mathbb U_n, \omega) = \mathcal A^2(\mathbb U_n; |F_{\pmb{\mu}}|^2)$, and the $\mathcal A^2(\mathbb U_n, \omega)$ norm of $f$ is bounded above and below by constant multiples of 
\[ \int_{\mathbb U_n} \bigl| f(\mathbf z) \bigr|^2  \times \bigl|F_{\pmb{\mu}}(\z)\bigr|^2 \, dV(\z).   \] Write $\pmb{\mu} = (\pmb{\mu'}, \pmb{\mu''}) \in\mathbb R^m \times \mathbb R^{n-m}$. Since $\mathbb U_n$ is the Cartesian product of a Reinhardt domain with a polydisk, and the weight $|F_{\pmb{\mu}}|^2$ is also of product type, the weighted Bergman space $\mathcal A^2(\mathbb U_n; |F_{\pmb{\mu}}|^2)$  is well-understood. A complete orthogonal basis for $\mathcal A^2(\mathbb U_n; |F_{\pmb{\mu}}|^2)$ is given by the set of monomial-type functions $\bigl\{F_{\mathbf k}(\z) F_{\pmb{\ell}} (\z''- \u'') : \k \in \mathbf K[\pmb{\mu'}], \; \pmb{\ell} \in (\mathbb N \cup \{0\})^{n-m}  \bigr\}$, where
\begin{align*}  \mathbf K[\pmb{\mu'}] &:= \bigl\{ \mathbf k \in \mathbb Z^m : F_{\k} \in \mathcal L^2 \bigl(\mathbb A_m(\mathbf 0'; \epsilon \mathbf 1'), |F_{\pmb{\mu}'}|^2 \bigr) \bigr\} \\ &\,= \{\mathbf k \in \mathbb Z^m : 2 \mathbf k + 2 \pmb{\mu'} + \mathbf 1' \text{ has  strictly positive entries} \}. 
\end{align*}  
Since $\omega$ is admissible, we know from the definition \eqref{def-admissible-weight} that $\mu_j < 1/2$ for all $1 \leq j \leq m$. Therefore any $\mathbf k = (k_1, \cdots, k_m) \in \mathbf K[\pmb{\mu'}]$ must obey $k_j > -\mu_j - 1/2 > -1$, and hence must have non-negative integer entries. 
Thus any $f \in \mathcal A^2(\mathbb U_n; \omega)$ is of the form 
\begin{equation} \label{f-form}  f(\mathbf z) = \sum_{\mathbf k}' F_{\mathbf k}(\mathbf z') f_{\mathbf k}(\mathbf z''), \text{ where } \mathbf z = (\mathbf z', \mathbf z'') \in \mathbb U_n,  \end{equation} 
where the sum $\sum'$ ranges over multi-indices $\k \in \mathbf K[\pmb{\mu}']$, and the functions $f_{\mathbf k}$ are analytic on $\mathbb D_{n-m}(\mathbf u''; \epsilon \mathbf 1'')$. The series converges both in $\mathcal L^2(\mathbb U_n; \omega)$ and also absolutely and uniformly over compact subsets of $\mathbb U_n$. Since the series \eqref{f-form} only admits non-negative integer powers of $\z'$, an application of the iterated Cauchy integral formula shows that such a function $f$ extends holomorphically to $\mathbb D_n(\mathbf u, \frac{\epsilon}{2} \mathbf 1)$, and hence to $\u$.
\end{proof} 

\subsection{Proof of Theorem \ref{Thm-Bergman-unweighted}}\quad
\begin{proof} 
\noindent {\bf{Part (\ref{omega1-admissible}):}} If $\mathcal A^2(\Omega_1^{\ast}, \omega_1) = \mathcal A^2(\Omega_1, \omega_1)$, then $B_{\Omega_1^{\ast}}( \cdot, \cdot; \omega_1) \equiv B_{\Omega_1}(\cdot, \cdot; \omega_1)$. Thus \eqref{Bergman-identity-2} and \eqref{Bergman-unweighted} follow respectively from \eqref{Bergman-kernel-identity} and \eqref{Bergman-diagonal}. The second statement in part \eqref{omega1-admissible} is a consequence of Proposition \ref{Bergman-star-prop}. 
\vskip0.1in
\noindent {\bf{Part (\ref{omega2-admissible}): }} Suppose now that $\omega_2$ is a weight of monomial type on $\Omega_2$ that is not necessarily admissible, say \[ \omega_2(\w) = \bigl| F_{\pmb{\nu}}(\w) \bigr|^2 \vartheta_2(\w) \text{ for some } \pmb{\nu} \in \mathbb R^n, \; \inf \bigl\{\w : \vartheta_2(\w) : \w \in \Omega_2 \bigr\} > 0.  \]  Given any $\chi \in \widehat{\mathbb G}_{\mathbf A}$, let $\b^{\ast} \in \mathbb Z^n$ be a vector such that $\varphi([\![\b^{\ast}]\!] = \chi$.  It follows from Lemma \ref{Lem3.5}(\ref{Lem3.5b}) that a vector $\b \in \mathbb Z^n$ has the same property if and only if $\b - \b^{\ast} = \m \cdot \mathbf A$ for $\m \in \mathbb Z^n$. We now compute the values of $\c$, as given by Theorem \ref{Thm1.2} \eqref{Thm2.2d},   corresponding to these two choices of vectors:
\[ \c(\b) -  \c(\b^{\ast}) = (\b^{\ast} - \b) \mathbf A^{-1} = - \m \]  
Choosing $\m \in \mathbb Z^n$ to have sufficiently large positive entries relative to $\c(\b^{\ast})$ and $\pmb{\nu}$, we can ensure that every entry of the vector $\c(\b) + \pmb{\nu}$ is $<\frac12$, so that the weight function $\eta_{\b}$ given by
\[ \eta_{\b} (\w) = \det(\mathbf A)^{-1} |F_{\c}(\w)|^2 \omega_2(\w) = \det(\mathbf A)^{-1}  |F_{\c + \pmb{\nu}}(\w)|^2 \vartheta_2(\w) \]
is admissible of monomial type on $\Omega_2$.  Set $\b_{\chi} = \b$. Invoking Proposition \ref{Bergman-star-prop} yields $\mathcal A^2(\Omega_2, \eta_{\b_{\chi}}) = \mathcal A^2(\Omega_2^{\ast}, \eta_{\b_{\chi}})$. Therefore the two spaces share the same Bergman kernel. Substituting $B_{\Omega_2}$ instead of $B_{\Omega_2^{\ast}}$ into \eqref{Bergman-kernel-identity} and \eqref{Bergman-diagonal} leads to the desired claim. 
\vskip0.1in
\noindent {\bf{Part (\ref{Omega-12}): }} This follows by combining parts (\ref{omega1-admissible}) and (\ref{omega2-admissible}), completing the proof. 
\end{proof} 

\section{Examples and applications}\label{appendix-section}

\noindent Here we apply the conclusions of this paper to a few specific domains and obtain Bergman kernel identities and/or estimates for these. 
\subsection {Example 1: Complex ellipsoids} \quad

\smallskip

Let 
\beas
\Omega_{1}&=\Big\{(z_{1},z_{2})\in \C^{2}:|z_{1}|^{2p}+|z_{2}|^{2q}<1\big\},&\omega_{1}(z_{1},z_{2})&\equiv 1,\\
\B_{2}&=\Big\{(\zeta_{1},\zeta_{2})\in \C^{2}:|\zeta_{1}|^{2}+|\zeta_{2}|^{2}<1\Big\}, &\omega_{2}(\zeta_{1},\zeta_{2})&\equiv 1.
\eeas
The domain $\Omega_{1}$ is an example of a \emph{complex ellipsoid} and $\B_{2}$ is the unit ball. If $\Phi_{\AAA}(\z) =(z_{1}^{p},z_{2}^{q})$ then $\Phi_{\AAA}:\Omega_{1}\to\B_{2}$ is a proper holomorphic mapping. In this case 
\begin{equation*}
\AAA=\left(\begin{matrix}p&0\\0&q\end{matrix}\right) \quad \text{  and  } \quad \AAA^{-1}=\left(\begin{matrix}p^{-1}&0\\0&q^{-1}\end{matrix}\right).
\end{equation*} 
Then $\G_{\AAA} \cong \mathbb Z_p \oplus \mathbb Z_q = \big\{(m_{1},m_{2})\in \Z^{2}:0\leq m_{1}\leq p-1,\,0\leq m_{2}\leq q-1\big\}$. The action of $\G_{\AAA}$ on $\C^{2}$ is given by \[ \xib(m_{1},m_{2})\otimes (z_{1},z_{2})=\left(e^{2\pi i \frac{m_{1}}{p}}z_{1}, e^{2\pi i \frac{m_{2}}{q}}z_{2}\right).\] If $\b=(b_{1},b_{2})\in \Z^{2}$, the action of the corresponding character $\chi_{\b}$ on $[\m]\in \G_{\AAA}$ is given by
\beas
\chi_{\b}([\m])=\exp\big[2\pi i \langle\m,\b\cdot\AAA^{-1}\rangle\big]=\exp\Big[2\pi i \Big(\frac{m_{1}b_{1}}{p}+\frac{m_{2}b_{2}}{q}\Big)\Big].
\eeas
 For $\chi \equiv \chi_{\b}$, the operator $\Pi_{\chi}$ defined in \eqref{Eqn1.12} becomes in this case  
\beas
\Pi_{\chi}[f](z_{1},z_{2})
&=
\frac{1}{pq}\sum_{[\m]\in \G_{\AAA}}
\exp\Big[2\pi i \Big(\frac{m_{1}b_{1}}{p}+\frac{m_{2}b_{2}}{q}\Big)\Big]\,\,
f \left(e^{2\pi i \frac{m_{1}}{p}}z_{1}, e^{2\pi i \frac{m_{2}}{q}}z_{2}\right).
\eeas
The property \eqref{invariance-Pi-chi} shows that the function
\beas
(z_1, z_2) \in \Omega_1 \mapsto \frac{z_{1}^{b_{1}}z_{2}^{b_{2}}}{pq}\sum_{[\m]\in \G_{\AAA}}
\exp\Big[2\pi i \Big(\frac{m_{1}b_{1}}{p}+\frac{m_{2}b_{2}}{q}\Big)\Big]\,
f \left(e^{2\pi i \frac{m_{1}}{p}}z_{1}, e^{2\pi i \frac{m_{2}}{q}}z_{2}\right) 
\eeas
is invariant under the action of $\G_{\AAA}$. Finally, according to Theorem \ref{Thm1.2}\eqref{Thm2.2d}, we compute \[ \c=(\1-\b)\cdot\AAA^{-1}-\1=\left(\frac{1-b_{1}-p}{p},\frac{1-b_{2}-q}{q} \right)=(c_{1},c_{2}), \] and so 
\beas
\eta_{\b}(\zeta_{1},\zeta_{2})=(pq)^{-1}|\zeta_{1}|^{2(1-b_{1}-p)p^{-1}}|\zeta_{2}|^{2(1-b_{2}-q)q^{-1}}.
\eeas 
Thus according to Theorem \ref{Thm-Bergman-unweighted},
\beas
B_{\Omega_{1}}(\z,\w)=\sum_{b_{1}=0}^{p-1}\sum_{b_{2}=0}^{q-1} (z_{1}\overline w_{1})^{-pb_{1}}(z_{2}\overline w_{2})^{-qb_{2}}B_{\B_{2}^{\ast}}\Big((z_{1}^{p},z_{2}^{q}),(w_{1}^{p},w_{2}^{q}); \eta_{\b}\Big).
\eeas
Thus the Bergman kernel for the complex ellipsoid in $\C^{2}$ can be written as a sum of weighted Bergman kernels in the punctured unit ball $\mathbb B_2^{\ast}$ of $\C^{2}$. For other formulas see for example \cite{Park}.

\subsection {Example 2: Variants of the Hartogs triangle}\label{Sec7.2} \quad

\smallskip

Let 
\beas
\Omega_{1}&=\Big\{(z_{1},z_{2})\in \C^{2}:0<|z_{1}|^{p}<|z_{2}|^{q}<1\big\},&\omega_{1}(z_{1},z_{2})&\equiv 1,\\
\Delta_{2}&=\Big\{(\zeta_{1},\zeta_{2})\in \C^{2}:|\zeta_{1}|<1,\,\,|\zeta_{2}|<1\Big\}, &\omega_{2}(\zeta_{1},\zeta_{2})&\equiv 1.
\eeas
$\Omega_{1}$ is a variant of the Hartog's triangle (where $p=q=1$) and $\Delta_{2}$ is the bidisk. If $\Phi_{\AAA}(z_{1},z_{2})=(z_{1}^{p}z_{2}^{-q}, z_{2}^{q})$ then $\Phi_{\AAA}:\Omega_{1}\to \Delta_{2}$ is a proper holomorphic map.  This time the action of $\G_{\AAA}$ on $\C^{2}$ is given by $\xib(m_{1},m_{2})\otimes (z_{1},z_{2})=\big(e^{2\pi i \frac{m_{1}+m_{2}}{p}}z_{1}, e^{2\pi i \frac{m_{2}}{q}}z_{2}\big)$, and if $\b=(b_{1},b_{2})\in \Z^{2}$, the action of the corresponding character on $[\m]\in \G_{\AAA}$ is given by 
\beas
\chi_{\b}([\m])=\exp\Big[2\pi i \Big(\frac{(m_{1}+m_{2})b_{1}}{p}+\frac{m_{2}b_{2}}{q}\Big)\Big].
\eeas
As in Example 1, we can write the Bergman kernel for $\Omega_{1}$ in terms of weighted Bergman kernels on the bidisk.

\subsection {Example 3: Complex monomial balls}\label{Sec7.3}\quad
\vskip0.1in
\noindent Let $\PP=\big\{\p_{1}, \ldots, \p_{d}\big\}\subset\Z^{n}$ be a spanning set of vectors in $\R^{n}$, each with non-negative integer entries. For $\a\in\C^{n}$, let us define 
 \beas
\B_{\PP}(\a,\mu) :=\Big\{\z\in\C^{n}:\sum_{j=1}^{d}|F_{\p_{j}}(\z)-F_{\p_{j}}(\a)|^{2}<\mu^{2}\Big\}.
 \eeas
We refer to $\B_{\PP}(\a,\mu)$ as a \emph{complex monomial ball} with center $\a$ and radius $\mu>0$. The study of such domains is part of a larger research program (see  \cite{{NP-2009}, {NP-Stein-conference}, {NP-preprint}, {NP-preprint-Reinhardt-general}, {NP-preprint-Reinhardt-monomial}}).  Using the results of this paper we obtain sharp estimates for $B_{\B_{\PP}(\a,\mu)}(\a,\a)$ (\emph{i.e.} diagonal estimates at the center) which are {\em{uniform}} in the parameter $\a$. 
\vskip0.1in
\noindent We briefly sketch how this is done. Note that if 
\beas
\Omega_{1}&=\big\{\z\in\C^{n}:\text{$|F_{\p_{j}}(\z)-F_{\p_{j}}(\a)|<d^{-\frac{1}{2}}\mu$ for $1\leq j \leq d$}\big\},\\
\Omega_{2}&= \big\{\z\in\C^{n}:\text{$|F_{\p_{j}}(\z)-F_{\p_{j}}(\a)|<\mu$ for $1\leq j \leq d$}\big\},
\eeas
then $\Omega_{1}\subset\B_{\PP}(\a,\mu)\subset\Omega_{2}$. It follows from  equation (\ref{Eqn1.4a}) that \bea\label{Eqn7.1}
B_{\Omega_{2}}(\a,\a)\leq B_{\B_{\PP}(\a,\mu)}(\a,\a)\leq B_{\Omega_{1}}(\a,\a)
\eea
and it suffices to obtain estimates at $(\a,\a)$ for the comparable domains $\Omega_{1}, \Omega_{2}$. Suppose that $\a=(a_{1}, \ldots, a_{n})\notin\H$ so that each $a_{j}\neq 0$, and let $\Psi_{\a}(z_{1}, \ldots, z_{n})=(a_{1}z_{1}, \ldots, a_{n}z_{n})$. Then
 \beas
 \Psi_{\a}^{-1}(\Omega_{1})&=\Big\{\z\in \C^{n}:\text{$|F_{\p_{j}}(\z)-1|<d^{-\frac{1}{2}}\mu\,|F_{\p_{j}}(\a)|^{-1}$ for $1\leq j \leq d$}\Big\},\\
\Psi_{\a}^{-1}(\Omega_{2})&=\Big\{\z\in \C^{n}:\text{$|F_{\p_{j}}(\z)-1|<\mu\,|F_{\p_{j}}(\a)|^{-1}$ for $1\leq j \leq  d$}\Big\}.
 \eeas
 Since $\Psi_{\a}$ is a biholomorphic mapping, it follows from equation (\ref{Eqn7.1}) that
 \beas
\Big (\prod_{j=1}^{n}a_{j}\Big)^{-2}B_{\Psi_{\a}^{-1}(\Omega_{1})}(1,1)\leq B_{\B_{\PP}(\a,\mu)}(\a,\a)\leq \Big(\prod_{j=1}^{n}a_{j}\Big)^{-2}B_{\Psi_{\a}^{-1}(\Omega_{2})}(1,1).
 \eeas
 This suggests that we obtain estimates for the Bergman kernel for domains of the form
 \bea
 \B_{\PP}(\vec\delta)=\Big\{\z\in\C^{n}:\text{$|F_{\p_{j}}(\z)-1|<\delta_{j}$ for $1 \leq j \leq d$}\Big\}
 \eea
 which are uniform in $\vec\delta=(\delta_{1}, \ldots, \delta_{d})\in (0,\infty)^{d}$. In \cite{NP-preprint}, we obtain a structure theorem for domains $\B_{\PP}(\vec\delta)$: after a monomial change of coordinates and depending on the components of $\vec{\delta}$ the domain $\mathbb B_{\mathcal P}$ is comparable to a Cartesian product of disks with small radius and axis deleted Reinhardt domains.  More precisely, there are absolute positive constants $c_0$ and $C_0$ so that for every $\vec{\delta} \in (0, \infty)^d$, there exist the following:
\begin{enumerate}[$\bullet$] 

\vskip0.1in
\item A linearly independent subset $\{\mathbf a_{i_{1}}, \cdots, \mathbf a_{i_{n}} \} \subseteq \mathcal P$ with the corresponding $n\times n$ matrix $\AAA\in\M_{n}(\Z)$.
\vskip0.1in
\item A partition $\{1, \ldots, d\}=J\cup K$ with $J=(j_{1}, \ldots, j_{m})$, $K=(k_{1}, \ldots, k_{d-m})$, and $J\cap K=\emptyset$.  Either $J$ or $K$ may be empty.
\vskip0.1in

\item a set $\mathcal R[K] = \{\mathbf r_j : j \in K \} \subseteq\mathbb Z^{n-m}$, 
\vskip0.1in
\end{enumerate} 
such that 
\beas
\D_{m}\big(c_{0}\vec\delta(J)\big)\times\W_{\RR}\big(c_{0}\vec\delta(K)\big)\subset\Phi_{\AAA}\big(\B_{\PP}(\vec\delta)\big)\subset \D_{m}\big(C_{0}\vec\delta(J)\big)\times\W_{\RR}\big(C_{0}\vec\delta(K)\big).
\eeas
Here $\vec{\delta}(J) = (\delta_{j_1}, \cdots, \delta_{j_m})$,  $\vec\delta(K)=(\delta_{k_{1}}, \ldots, \delta_{k_{d-m}})$, and 
\begin{align}
\mathbb D_m( \lambda \vec{\delta}(J)) &= \bigl\{ (w_1, \cdots, w_m)\in \mathbb C^m : |w_j - 1| < \lambda \delta_{j} \text{ for all } j \in J \bigr\}, \text{ and } \label{disk} \\
\mathbb W_{\mathcal R}^{\ast}(\lambda \vec{\delta}(I))  &= \bigl\{ (w_{m+1}, .., w_n) \in \mathbb C^{n-m} : 0 < |F_{\mathbf r_j} (w_{m+1}, .., w_n)| < \lambda \delta_j \text{ for all } j \in I \bigr\}. \label{Reinhardt}  
\end{align}  
The domain $\mathbb W_{\mathcal R}$ is a Reinhardt domain which may or may not be axes-deleted. The papers \cite{NP-preprint-Reinhardt-general} and \cite{NP-preprint-Reinhardt-monomial} provide geometric estimates 
and a computationally effective algorithm for obtaining sharp estimates for the Bergman kernel on the diagonal for Reinhardt domains of the form $\mathbb W_{\mathcal R}$ and $\mathbb W_{\mathcal R}^{\ast}$. Since the Bergman kernel of the polydisk $\D_{m}$ is well understood, the structure theorem and the results of this paper provide the desired estimates.

\vskip0.1in
\noindent We show how this procedure works in a simple case. Let $n=2$ and $\PP=\big\{(1,0), (0,1), (1,1)\big\}$ so that
\bea\label{Eqn7.5}
\B_{\PP}(\vec\delta)=\Big\{(z_{1},z_{2})\in \C^{2}: |z_{1}-1|<\delta_{1},\,\, |z_{2}-1|<\delta_{2},\,\,|z_{1}z_{2}-1|<\delta_{3}\Big\}.
\eea
The nature of $\B_{\PP}(\vec\delta)$ depends on the sizes of $\delta_{1}, \delta_{2}, \delta_{3}$. Consider the case in which $\delta_{1}$ and $\delta_{2}$ are large and $\delta_{3}$ is small.  Explicitly suppose that 
\begin{equation} \label{Eqn7.6}
\frac{3}{2}<\delta_{1}<\infty, \quad \frac{3}{2}<\delta_{2}<\infty, \quad 0<\delta_{3}<\frac{1}{2}.
\end{equation}
Then 
\[
\Big\{|z_{1}|<\frac{\delta_{1}}{3},\,\, |z_{2}|<\frac{\delta_{2}}{3},\,\,|z_{1}z_{2}-1|<\delta_{3}\Big\}
\subset
\B_{\PP}(\vec\delta)
\subset
\Big\{|z_{1}|<3\delta_{1},\,\, |z_{2}|<3\delta_{2},\,\,|z_{1}z_{2}-1|<\delta_{3}\Big\},
\]
 so $\B_{\PP}(\vec\delta)$ is comparable to \[ \B_{\PP}^{\prime}(\vec\delta)=\big\{(z_{1}, z_{2})\in \C^{2}:|z_{1}|<\delta_{1}, \,\,|z_{2}|<\delta_{2},\,\,|z_{1}z_{2}-1|<\delta_{3}\big\}. \]
Let $\Phi_{\AAA}(z_{1},z_{2})=(z_{1}z_{2},z_{2})$, so \[ \Phi_{\mathbf A}(\mathbb B'_{\mathcal P}(\vec\delta)) =\Big\{(w_{1},w_{2}):|w_{1}w_{2}^{-1}|<\delta_{1}, |w_{2}|<\delta_{2},\,|w_{1}-1|<\delta_{3}\Big\}. \]  If $(z_{1},z_{2})\in \B_{\PP}^{\prime}(\vec\delta)$ then $ \frac{1}{2} < 1 - \delta_3 < |z_1 z_2| < 1 + \delta_3 < \frac{3}{2}$ and so if $(w_{1},w_{2})\in \Phi_{\AAA}\big(\B_{\PP}(\vec\delta)\big)$ then $\frac{1}{2}<|w_{1}|<\frac{3}{2}$. Thus $\Phi_{\mathbf A}(\mathbb B'_{\mathcal P})$
is comparable to \[ \big\{(w_{1},w_{2}):|w_{1}-1|<\delta_{3},\,\,\delta_{1}^{-1}<|w_{2}|<\delta_{2}\big\}=\D\times\W. \]
Here $\D$  is a disk in $\C$ of radius $\delta_{3}$ and $\W$ is an annulus with inner and outer radii $\delta_{1}^{-1}$ and $\delta_{2}$ respectively. It follows that under the size hypotheses in (\ref{Eqn7.6}), we have
\beas
c\,\delta_{3}^{-1}\log(\delta_{1}\delta_{2})<B_{\B_{\PP}(\vec\delta)}(\1,\1)<C\,\delta_{3}^{-1}\log(\delta_{1}\delta_{2})
\eeas
where the constants $c, C$ are independent of $\vec\delta$.

}

\end{document}